\documentclass[11pt]{article}
\usepackage{amsmath,amssymb,amsthm,txfonts,amsbsy}
\usepackage{authblk}
\usepackage{graphicx}
\usepackage{color}
\usepackage{comment}
\usepackage[graph,color]{xy}
\definecolor{bcP}{rgb}{1,0,1}

\usepackage{tikz, framed}
\usetikzlibrary{arrows,shapes,matrix,decorations.pathmorphing}
\usetikzlibrary{backgrounds}

\theoremstyle{plain} 
\newtheorem{Theorem}{Theorem} 

\newtheorem{Definition}[Theorem]{Definition}

\newtheorem{Lemma}[Theorem]{Lemma}
\newtheorem{Proposition}[Theorem]{Proposition}
\newtheorem{Corollary}[Theorem]{Corollary}
\newtheorem{Conjecture}[Theorem]{Conjecture}
\theoremstyle{definition}
\newtheorem{Example}[Theorem]{Example}

\theoremstyle{remark} 
\newtheorem{examplex}[Theorem]{Example}
\newenvironment{example}
  {\pushQED{\qed}\examplex}
  {\popQED\endexamplex}

\newtheorem{Remark}[Theorem]{Remark}

  \newcommand{\kS}{\mathcal{S}}

 \newcommand{\FF}{\mathbb{F}}

\newcommand{\kC}{\mathcal{C}}
\newcommand{\kL}{\mathcal{L}}

\title{The direct sum of $q$-matroids}
\author[1]{Michela Ceria}
\affil[1]{\small 
Dept. of Mechanics, Mathematics \& Management, Politecnico di Bari,  Via Orabona 4 - 70125 Bari - Italy; michela.ceria@gmail.com}

\author[2]{Relinde Jurrius}
\affil[2]{\small Faculty of Military Sciences, Netherlands Defence Academy, The Netherlands; rpmj.jurrius@mindef.nl}

\date{}
\begin{document}
\newpage
\fussy
\maketitle

\begin{abstract}
For classical matroids, the direct sum is one of the most straightforward methods to make a new matroid out of existing ones. This paper defines a direct sum for $q$-matroids, the $q$-analogue of matroids. This is a lot less straightforward than in the classical case, as we will try to convince the reader. With the use of submodular functions and the $q$-analogue of matroid union we come to a definition of the direct sum of $q$-matroids. As a motivation for this definition, we show it has some desirable properties.  
\end{abstract}

\section{Introduction}

The study of $q$-matroids, introduced by Crapo \cite{crapo1964theory}, has recently attracted renewed attention because of its link to network coding. After the reintroduction of the object by Jurrius and Pellikaan \cite{JP18} and independently that of $(q,r)$-polymatroids by Shiromoto \cite{S19}, several other papers have studied these objects, often in relation to rank metric codes. See for example \cite{BCR21,GJ20,ghorpade_shellability_2022,gluesing2021qpolyindep,gluesing2021qpoly,gorla2019rank,panja2019some}. \\  

Roughly speaking, a $q$-analogue in combinatorics is a generalisation from sets to finite dimensional vector spaces. So a $q$-matroid is a finite dimensional vector space with a rank function defined on its subspaces, satisfying certain properties. One can also view this generalisation from the point of view of the underlying lattice: where matroids have the Boolean lattice (of sets and subsets) as their underlying structure, $q$-matroids are defined over the subspace lattice. The work of finding a $q$-analogue often comes down to writing a statement about sets in such a lattice-theoretic way that the $q$-analogue is a direct rephrasing for the subspace lattice. However, this is often not a trivial task, for two reasons. First, there might be several equivalent ways to define something over the Boolean lattice, where the $q$-analogues of these statements are not equivalent. Secondly, some statements on the Boolean lattice do not have a $q$-analogue: the subspace lattice is, contrarily to the Boolean lattice, not distributive. \\

In this paper we consider the {\em direct sum} of two $q$-matroids. An option to do this is to extend to the realm of sum-matroids \cite{panja2019some}, but we are looking for a construction that gives a $q$-matroid. This is one of the cases as mentioned above where the $q$-analogue is a lot harder than the relatively simple procedure of taking the direct sum of two classical matroids. The latter is defined as follows. Let $E_1$ and $E_2$ be disjoint sets and let $E=E_1\cup E_2$. Let $M_1=(E_1,r_1)$ and $M_2=(E_2,r_2)$ be two matroids. Then the direct sum $M_1\oplus M_2$ is a matroid with ground set $E$. For its rank function, note that we can write any $A\subseteq E$ as a disjoint union $A=A_1\sqcup A_2$ with $A_1\subseteq E_1$ and $A_2\subseteq E_2$. The rank function of the direct sum $M_1\oplus M_2$ is now given by $r(A)=r_1(A_1)+r_2(A_2)$. \\
If we try to mimic this procedure in the $q$-analogue, we run into trouble quite fast. Let $E_1$ and $E_2$ be disjoint subspaces and let $E=E_1\oplus E_2$. If we consider a subspace $A\subseteq E$, it might be that we cannot write it as a direct sum $A_1\oplus A_2$, with $A_1\subseteq E_1$ and $A_2\subseteq E_2$. In fact, most of the subspaces of $E$ can not be written in this way. Our goal is to define a rank function for these subspaces. \\
A naive try is to define a rank function in the $q$-analogue for all spaces $A\subseteq E$ that can be written as $A_1\oplus A_2$, and hope that the axioms for the rank function take care of the rest of the spaces. However, as we show with an example in Section \ref{ExDim4}, this procedure does not give us a unique direct sum. As a byproduct of this example, we find the smallest non-representable $q$-matroid. \\

Our solution for the direct sum of $q$-matroids is the following. We will first define the notion of {\em matroid union} for $q$-matroids in Section \ref{MatUn}. This notion is dual to matroid intersection, that we consider in Section \ref{IntDual}. Then we show in Section \ref{DirSum} that the direct sum of a $q$-matroid and a loop can be defined. Finally, we define the direct sum of two $q$-matroids by first adding loops to get two $q$-matroids on the same ground space, and then taking their matroid union. \\
To motivate this definition we show that this construction has several desirable properties. First of all, it generalises our naive attempt in Section \ref{ExDim4}. Also, taking the dual of a direct sum is isomorphic to first taking duals and then taking the direct sum. Lastly, restriction and contraction to $E_1$ and $E_2$ give back one of the original $q$-matroids. \\

We finish this paper by briefly considering what it would mean for a $q$-matroid to be connected (Section \ref{Connect}). As one might assume from the difficulty of the direct sum, this is also not an easy endeavour. We outline the problems that appear when trying to make a $q$-analogue of some of the several equivalent definitions of connectedness in classical matroids. \\

At the end of this paper (Appendix \ref{qcatalogue}) we give a catalogue of small $q$-matroids. In the paper, we will often refer to examples from this catalogue. Since the study of $q$-matroids is a relatively new one, we hope this catalogue to be useful for others learning about $q$-matroids.

\section{Preliminaries}\label{PrelimSec}

Following the notation of \cite{BCR21} we denote by $n$ a fixed positive integer and by $E$ a fixed $n$-dimensional vector space over an arbitrary field $\mathbb{F}.$ The notation $\mathcal{L}(E)$ indicates the \textbf{lattice of subspaces} of $E$. For any $A,B\in\mathcal{L}(E)$ with $A\subseteq B$ we denote by $[A,B]$ the interval between $A$ and $B$, that is, the lattice of all subspaces $X$ with $A\subseteq X\subseteq B$. For $A\subseteq E$ we use the notation $\mathcal{L}(A)$ to denote the interval $[\{0\},A]$. For more background on lattices, see for example Birkhoff \cite{birkhoff}. \\

We use the following definition of a $q$-matroid.

\begin{Definition}\label{rankfunction}
A $q$-matroid $M$ is a pair $(E,r)$ where $r$ is an integer-valued  function defined on the subspaces of $E$ with the following properties:
\begin{itemize}
\item[(R1)] For every subspace $A\in \kL(E)$, $0\leq r(A) \leq \dim A$. 
\item[(R2)] For all subspaces $A\subseteq B \in \kL(E)$, $r(A)\leq r(B)$. 
\item[(R3)] For all $A,B\in \kL(E)$, $r(A+ B)+r(A\cap B)\leq r(A)+r(B)$.  
\end{itemize}
The function $r$ is called the \textbf{rank function} of the $q$-matroid. 
\end{Definition}

Sometimes, we will need to deal with the rank functions of more than one $q$-matroid at a time, say $M,M'$, with ground spaces $E$, $E'$, respectively. In order to distinguish them (and emphasize the $q$-matroid in which we are computing the rank), we will write $r(M;A)$ for the rank in $M$ of a subspace $A\subseteq E$ and $r(M';A')$ for the rank in $M'$ of a subspace $A'\subseteq E'$. For a $q$-matroid $M$ with ground space $E$, we use $r(M)$ as notation for $r(M;E)$. \\
We will use the axioms of the rank functions repeatedly in our proofs, as well as the following lemma that follows by induction from the axiom (R2') in \cite[Theorem 31]{BCR21}.

\begin{Lemma}[Local semimodularity]\label{localSemimodularity}
If $A\subseteq B\subseteq E$ then $r(B)-r(A)\leq\dim B-\dim A$.
\end{Lemma}

A way to visualise a $q$-matroid is by taking the Hasse diagram of the underlying subspace lattice and colour all the covers: red if the rank goes up and green if the rank stays the same. This is done in Appendix \ref{qcatalogue}. More properties of this bi-colouring can be found in \cite{BCJ17}.

There are several important subspaces in $q$-matroids.

\begin{Definition}
Let $(E,r)$ be a $q$-matroid. A subspace $A$ of $E$ is called an \textbf{independent} space of $(E,r)$ if \[r(A)=\dim A.\]
An independent subspace that is maximal with respect to inclusion is called a \textbf{basis}.
A subspace that is not an independent space of $(E,r)$ is called a \textbf{dependent space} of the $q$-matroid $(E,r)$. 
We call $C \in \kL(E)$ a \textbf{circuit} if it is itself a dependent space and every proper subspace of $C$ is independent.
A \textbf{spanning space} of the $q$-matroid $(E,r)$ is a subspace $S$  such that $r(S)=r(E)$. 
A subspace $A$ of a $q$-matroid $(E,r)$ is called a \textbf{flat} if for all $1$-dimensional subspaces $x \in \kL(E)$ such that $x\nsubseteq A$ we have \[r(A+x)>r(A).\]
A subspace $H$ is called a $\textbf{hyperplane}$ if it is a maximal proper flat, i.e., if $H \neq E$ and the only flat that properly contains $H$ is $E$.
A $1$-dimensional subspace $\ell$ is called a \textbf{loop} if $r(\ell)=0$. All loops together form a subspace (\cite[Lemma 11]{JP18}) that we call the \textbf{loop space} of $M$.
\end{Definition}

A $q$-matroid can be equivalently defined by its independent spaces, bases, circuits, spanning spaces, flats and hyperplanes. See \cite{BCR21} for an overview of these cryptomorphic definitions. We will explicitly use the axioms for circuits:

\begin{Definition}\label{circuit-axioms}
Let $\mathcal{C}\subseteq\mathcal{L}(E)$. We
define the following \textbf{circuit axioms}.
\begin{itemize}
\item[(C1)] $\{0\}\notin\mathcal{C}$.
\item[(C2)] For all $C_1,C_2\in\mathcal{C}$, if $C_1\subseteq C_2$, then $C_1=C_2$.
\item[(C3)]  For distinct $C_1,C_2 \in \kC$ and any $X\in \kL(E)$ of codimension $1$ there is a circuit $C_3 \in \kC$ such that $C_3 \subseteq (C_1+C_2)\cap X$.
\end{itemize}
If $\kC$ satisfies the circuit axioms (C1)-(C3), we say that $(E,\mathcal{C})$ is a \textbf{collection of circuits}.
\end{Definition}

Recall that a lattice isomorphism between a pair of lattices $(\kL_1,\leq_1,\vee_1,\wedge_1)$ and $(\kL_2,\leq_2,\vee_2,\wedge_2)$ is a bijective function $\varphi:\kL_1\longrightarrow\kL_2$ that is order-preserving and preserves the meet and join, that is, for all $x,y\in\kL_1$ we have that $\varphi(x\wedge_1 y)=\varphi(x)\wedge_2\varphi(y)$ and $\varphi(x\vee_1 y)=\varphi(x)\vee_2\varphi(y)$. A lattice anti-isomorphism between a pair of lattices is a bijective function $\psi:\kL_1\longrightarrow\kL_2$ that is order-reversing and interchanges the meet and join, that is, for all $x,y\in\kL_1$ we have that $\psi(x\wedge_1 y)=\psi(x)\vee_2\psi(y)$ and $\psi(x\vee_1 y)=\psi(x)\wedge_2\psi(y)$.  
We hence define a notion of equivalence and duality between $q$-matroids.

\begin{Definition}
Let $E_1,E_2$ be vector spaces over the same field $\mathbb{F}$. Let $M_1=(E_1,r_1)$ and $M_2=(E_2,r_2)$ be $q$-matroids. We say that $M_1$ and $M_2$ are \textbf{lattice-equivalent} or \textbf{isomorphic} if there exists a lattice isomorphism $\varphi:\kL(E_1)\longrightarrow \kL(E_2)$ such that $r_1(A)=r_2(\varphi(A))$ for all $A\subseteq E_1$. In this case we write $M_1 \cong M_2$.
\end{Definition}
 
Fix an anti-isomorphism $\perp:\kL(E)\longrightarrow\kL(E)$ that is an involution. For any subspace $X \in \kL(E)$ we denote by $X^\perp$ the \textbf{dual} of $X$ in $E$ with respect to $\perp$. Note that since an anti-isomorphism preserves the length of intervals, we have for any $X\leq\kL(E)$ that $\dim(X^\perp)=\dim(E)-\dim(X)$. \\
From a lattice point of view, if $B=B_1\oplus B_2$, then $B=B_1\vee B_2$ and $B_1\wedge B_2=0$. Since $\perp$ is an anti-isomorphism of $\kL(E)$, we have that $B^\perp=B_1^\perp\wedge B_2^\perp$ and $B_1^\perp\vee B_2^\perp=1$. Important operations on $q$-matroids are restriction, contraction and duality. We give a short summary here and refer to \cite{BCIR21,JP18} for details.

\begin{Definition}\label{defdual}
Let $M=(E,r)$ be a $q$-matroid. Then $M^*=(E,r^*)$ is also a $q$-matroid, called the \textbf{dual $q$-matroid}, with rank function
\[ r^*(A)=\dim(A)-r(E)+r(A^\perp). \]
The subspace $B$ is a basis of $M$ if and only if $B^\perp$ is a basis of $M^*$. From bi-colouring point of view, we get the dual $q$-matroid by turning the Hasse diagram upside down and interchange all red and green covers.
\end{Definition}

\begin{Definition}\label{restr}
Let $M=(E,r)$ be a $q$-matroid. The \textbf{restriction} of $M$ to a subspace $X$ is the $q$-matroid $M|_X$ with ground space $X$ and rank function $r_{M|_X}(A)=r_M(A)$.
The \textbf{contraction} of $M$ of a subspace $X$ is the $q$-matroid $M/X$ with ground space $E/X$ and rank function $r_{M/X}(A)=r_M(A)-r_M(X)$. A $q$-matroid that is obtained by restriction and contraction of $M$ is called a \textbf{minor} of $M$.
\end{Definition}

\begin{Theorem}\label{DualRestrContr}
Restriction and contraction are dual operations, that is, $M^*/X\cong  (M|_{X^\perp})^*$ and $(M/X)^* \cong M^*|_{X^\perp}$.
\end{Theorem}

Finally, we will define what it means for a $q$-matroid to be representable and give an example of an important class of $q$-matroids.

\begin{Definition}
Let $M=(E,r)$ be a $q$-matroid of rank $k$ over a field $K$. Let $A\subseteq E$ and let $Y$ be a matrix with column space $A$. We say that $M$ is \textbf{representable} if there exists a $k\times n$ matrix $G$ over an extension field $L/K$ such that $r(A)$ is equal to the matrix rank of $GY$ over $L$.
\end{Definition}

\begin{Example}
Let $k$ be a positive integer, $k \leq n$. 
The \textbf{uniform $q$-matroid} is the $q$-matroid $M=(E,r)$  with rank function defined as follows:
\[ r(U):= \left\{ \begin{array}{cl}
	\dim(U) & \text{ if } \dim(U) \leq k, \\
	k & \text{ if } \dim(U) > k.
\end{array}\right.
\] 
We denote this $q$-matroid by $U_{k,n}$.
\end{Example}

\section{Intuitive try for the direct sum}\label{ExDim4}

As stated in the introduction, the $q$-analogue of the direct sum is not straightforward. Let $E=E_1\oplus E_2$ be a direct sum of subspaces and let $A\subseteq E$. Then we cannot, in general, decompose $A\subseteq E$ as $A=A_1\oplus A_2$ with $A_1\subseteq E_1$ and $A_2\subseteq E_2$. \\
With other cryptomorphic definitions of $q$-matroids we run into similar problems. Look for example at the independent spaces. In the classical case, the independent sets of the direct sum $M_1\oplus M_2$ are the unions of an independent set in $M_1$ and an independent set in $M_2$. If we want to take the direct sum of the $q$-matroids $M_1=U_{1,1}$ and $M_2=U_{1,1}$, we expect all subspaces to be independent. However, not all such spaces can be written as the sum of an independent space in $M_1$ and an independent space in $M_2$. Similar problems arise when trying to construct the bases and circuits of the direct sum of the $q$-matroids $M_1$ and $M_2$.
\\
In this section we explore if we can define the rank function of a direct sum of $q$-matroids by simply defining $r(A)=r_1(A_1)+r_2(A_2)$ for all $A$ that can be written as $A=A_1\oplus A_2$, and hoping that the rank axioms will take care of the rest of the subspaces. (Spoiler alert: it will not work).

\subsection{First definition and properties}\label{FirstDef}

Let us make our first trial to define the direct sum. We start with a definition mimicking the classical case. We consider these properties desirable for the direct sum of $q$-matroids. We also prove some direct consequences of these properties. The properties from Definition \ref{def-directsum1} will turn out not to define a unique $q$-matroid, hence they are not sufficient for defining the direct sum of $q$-matroids. However, our final definition will satisfy these properties.

\begin{Definition}\label{def-directsum1}
Let $M_1=(E_1,r_1)$ and $M_2=(E_2,r_2)$ be two $q$-matroids on trivially intersecting ground spaces. For a $q$-matroid $M=(E,r)$ on the ground space $E=E_1\oplus E_2$ we define the following properties:
\begin{itemize}
\item the minors $M|_{E_1}$ and $M/E_2$
are both isomorphic to $M_1$,
\item the minors $M|_{E_2}$ and $M/E_1$
are both isomorphic to $M_2$.
\end{itemize}
\end{Definition}

In particular, it follows from this construction that the rank of $M$ is the sum of the ranks of $M_1$ and $M_2$. The next theorem shows that this definition is equivalent to what we recognise as the $q$-analogue of the definition of direct sum in the classical case.

\begin{Theorem}\label{RangoSomma}
Let $M_1=(E_1,r_1)$ and $M_2=(E_2,r_2)$ be two $q$-matroids on trivially intersecting ground spaces. We define a $q$-matroid $M=(E,r)$ on the ground space $E=E_1\oplus E_2$. Then $M$ satisfies the properties of Definition \ref{def-directsum1} if and only if for each $A\subseteq E_1$ and $B\subseteq E_2$ it holds $r(A+B)=r_1(A)+r_2(B)$.
\end{Theorem}
\begin{proof}
First, assume $M$ satisfies the properties of Definition \ref{def-directsum1}. Note that for all $A\subseteq E_1$ we have $r(A)=r(M;A)=r(M|_{E_1};A)=r_1(A)$ and similarly, for all $B\subseteq E_2$ we have $r(B)=r_2(B)$. So we need to show that $r(A+B)=r(A)+r(B)$. We prove this by applying semimodularity multiple times. First we apply it to $A$ and $B$. Since $A\cap B=\{0\}$, we have $r(A\cap B)=0$ and (r3) gives us
\[ r(A+B)\leq r(A)+r(B). \]
We claim that $r(M;E_1+B)=r(M;E_1)+r(M;B)$.  
Indeed, $r(M/E_1;(E_1+B)/E_1)=r(M_2;B)$ by Definition \ref{def-directsum1}.
Moreover, by Definition \ref{restr}, $r(M/E_1;(E_1+B)/E_1)=r(M; E_1+B)-r(M;E_1)$. Summing up
\[r(M;B)=r(M_2;B)=r(M/E_1;(E_1+B)/E_1)=r(M; E_1+B)-r(M;E_1),\]
so $r(M;E_1+B)=r(M;E_1)+r(M;B)$. Now we apply (r3) to $E_1$ and $A+B$.
\begin{align*}
r(E_1)+r(A+B) & \geq r(E_1+(A+B))+r(E_1\cap(A+B)) \\
 & = r(B+E_1)+r(A) \\
 & = r(B)+r(E_1)+r(A).
\end{align*}
This implies that
\[ r(A+B)\geq r(A)+r(B). \]
Combining the two inequalities gives the desired equality: $r(A+B)=r(A)+r(B)$. \\
For the other implication, suppose that $r(A+B)=r_1(A)+r_2(B)$. The two conditions in Definition \ref{def-directsum1} are symmetric, so we only need to prove the first one. We show that the rank function on $M|_{E_1}$ is equal to the rank function on $M_1$. Let $A\subseteq E_1$. Then
\[ r(M|_{E_1};A)=r(M;A)=r_1(A)+r_2(0)=r_1(A). \]
Now for $M/E_2$, let $C\subseteq E$ such that $E_2\subseteq C$. Then we can write $C=A+E_2$ with $A\subseteq E_1$. Then
\[ r(M/E_2;C/E_2)=r(M;C)-r(M;E_2)=r_1(A)+r_2(E_2)-r(M;E_2)=r_1(A). \]
It follows that $M|_{E_1}$ and $M/E_2$ are both isomorphic to $M_1$.
\end{proof}

As mentioned, the classical case of this last theorem is exactly the definition of the rank in the direct sum of matroids. This implies that Definition \ref{def-directsum1}, when applied to the classical case, completely determines the direct sum. We will see in the next subsection that this is not the case in the $q$-analogue. \\

We will close this section with some small results that show that Definition \ref{def-directsum1} implies the rank of all spaces of dimension and codimension $1$. Note that the next results only depend on Definition \ref{def-directsum1}, with the exception of Lemma \ref{allredontop}.

\begin{Proposition}\label{prop-noloops}
Let $M$ be a $q$-matroid satisfying the properties of Definition \ref{def-directsum1}. Suppose $M_1$ has loop space $L_1$ and $M_2$ has loop space $L_2$. Then the loop space of $M$ is $L_1\oplus L_2$.
\end{Proposition}
\begin{proof}
Since loops come in subspaces \cite[Lemma 11]{JP18}, $L_1\oplus L_2$ in $E$ only contains loops. We will show $M$ contains no other loops.
Suppose, towards a contraction, that there is a loop $\ell$ in $M$ that is not in $L_1\oplus L_2$. By assumption, $\ell$ is not in $E_1$ or in $E_2$. First we apply the semimodular inequality to $E_1$ and $\ell$:
\begin{align*}
r(E_1+\ell)+r(E_1\cap\ell) & \leq r(E_1)+r(\ell) \\
r(E_1+\ell)+0 & \leq r(E_1)+0
\end{align*}
hence $r(E_1+\ell)=r(E_1)$. Now we consider the $1$-dimensional space $x=(E_1+\ell)\cap E_2$. We claim that this space has rank $1$. Towards a contradiction, suppose $r(x)=0$ hence $x\subseteq L_2$. Then $r(\ell+x)=0$. Let $y$ be the $1$-dimensional space $(\ell+x)\cap E_1$. It has rank $0$ because it is in $\ell+x$, hence $y\subseteq L_1$. Now we have $\ell\subseteq y+x\subseteq L_1\oplus L_2$, which is a contradiction to $\ell\not\subseteq L_1\oplus L_2$. We conclude that $r(x)=1$. \\
Now we apply the semimodular inequality to $E_1+\ell$ and $E_2$.
\begin{align*}
r((E_1+\ell)+ E_2)+r((E_1+\ell)\cap E_2) & \leq r(E_1+\ell)+r(E_2) \\
r(E_1+E_2)+1 & \leq r(E_1)+r(E_2)
\end{align*}
and this is a contradiction. So there are no loops outside $L_1\oplus L_2$ in $M$.
\end{proof}

In particular, since we know exactly what are the loops of the direct sum, we know that all other $1$-dimensional spaces have rank $1$. Dually, we can derive a similar result for the codimension-$1$ spaces.

The next Lemma holds for all $q$-matroids. It is the dual of the statement that loops come in subspaces.

\begin{Lemma}\label{allredontop}
Let $M=(E,r)$ be a $q$-matroid. Let $H$ be the intersection of all codimension $1$ spaces in $E$ of rank $r(M)-1$. Then the spaces $A$ such that $H\subseteq A\subseteq E$ are exactly all the elements of $\kL(E)$ such that $r(E)-r(A)=\dim E-\dim A$.
\end{Lemma}
\begin{proof}
Let $X$ be a codimension $1$ space such that $r(X)=r(E)-1$. Consider the dual $q$-matroid $M^*$. Then $r^*(X^\perp)=\dim X^\perp-r(E)+r(X)=1-r(E)+r(E)-1=0$. Hence $X^\perp$ is a loop in $M^*$. This implies that $H^\perp$ is the sum of all loops in $M^*$, hence it is the loop space of $M^*$ and there are no other loops in $M^*$. For any $A$ such that $H\subseteq A\subseteq E$ we have that $A^\perp\subseteq H^\perp$, so $A^\perp$ has rank $0$ in $M^*$. This implies
\begin{align*} r(E)-r(A) & =r(E)-(r^*(A^\perp)+\dim A-r^*(E)) \\
 & =r(E)+r^*(E)-\dim A \\
 & =\dim E-\dim A.
\end{align*}
Conversely, if $A$ is a subspace such that $r(E)-r(A)=\dim E-\dim A$, then $r^*(A)=0$ by the same calculation as above. This implies the only spaces $A$ for which it holds that $r(E)-r(A)=\dim E-\dim A$, are the spaces such that $H\subseteq A\subseteq E$.
\end{proof}

The next result is the dual of Proposition \ref{prop-noloops}.

\begin{Proposition}\label{prop-nocoloops}
Let $M$ be a $q$-matroid satisfying the properties of Definition \ref{def-directsum1}. Suppose $M_1$ and $M_2$ do not have any codimension $1$ spaces of rank $r(M_1)-1$ and $r(M_2)-1$, respectively. Then $M$ does not have any codimension $1$ spaces of rank $r(M)-1$.
\end{Proposition}
\begin{proof}
Suppose, towards a contraction, that there is a codimension $1$ space $H$ of rank $r(M)-1$ in $M$. By construction, $H$ does not contain $E_1$ or $E_2$. So $E_1\cap H$ is of codimension $1$ in $E_1$, and by construction it has rank $r(E_1)$. \\
Now we apply the semimodular inequality to $E_1$ and $H$.
\begin{align*}
r(E_1\cap H)+r(E_1+H) & \leq r(E_1)+r(H) \\
r(E_1)+r(M) & \leq r(E_1)+r(M)-1
\end{align*}
and this is a contradiction. So there are no codimension $1$ spaces in of rank $r(M)$ in $M$.
\end{proof}

\subsection{Non-uniqueness of the first definition}\label{subsec-NotUnique}

In this section we show by example that Definition \ref{def-directsum1} does not uniquely define the direct sum of $q$-matroids. \\

Let $E=\mathbb{F}_2^4$ and let $M_1=M_2=U_{1,2}$. We will attempt to construct the direct sum $M=M_1\oplus M_2$. We assume that it has the properties from Definition \ref{def-directsum1}. So the $q$-matroid $M$ has at least two circuits: $E_1$ and $E_2$. Our goal is to determine $M$ completely. Note that Theorem \ref{RangoSomma} defines the rank for all subspaces of $E$ that can be written as a direct sum of a subspace of $E_1$ and a subspace of $E_2$. \\

All $1$-dimensional spaces in $E$ have rank $1$ because of Proposition \ref{prop-noloops} and by Proposition \ref{prop-nocoloops} all $3$-dimensional spaces in $E$ have rank $2$. This means that what is left to do is to decide for all $2$-dimensional spaces if they have rank $1$ or rank $2$, that is, whether they are a circuit or an independent space. We use the next lemma for this.

\begin{Lemma}\label{CircuitIntersection}
Let $M_1=M_2=U_{1,2}$ and let $M$ satisfy the conditions of Definition \ref{def-directsum1}. Let $C_1$ and $C_2$ be circuits of $M$ of dimension $2$. Then $\dim(C_1\cap C_2)\neq1$.
\end{Lemma}
\begin{proof}
If $C_1=C_2$, the result is clear. So let $C_1\neq C_2$. Suppose, towards a contradiction, that $\dim(C_1\cap C_2)=1$. Then $\dim(C_1+C_2)=3$. Now apply semimodularity to $C_1$ and $C_2$.
\begin{align*}
r(C_1+C_2)+r(C_1\cap C_2) & \leq r(C_1)+r(C_2) \\
2+1 & \leq 1+1
\end{align*}
This is a contradiction, hence $\dim(C_1\cap C_2)\neq1$.
\end{proof}

This means that every $2$-dimensional space that intersects with either $E_1$ or $E_2$ is independent. A counting argument shows that there are only six $2$-dimensional spaces that have trivial intersection with both $E_1$ and $E_2$. Denote by $A$, $B$, $C$, $D$, $F$, $G$ the six $2$-spaces of unknown rank. The following is independent of a choice of basis for $E$, but for convenience, we can coordinatize the spaces in the following way:
\[ E_1=\left\langle \begin{array}{cccc} 1 & 0 & 0 & 0 \\ 0 & 1 & 0 & 0 \end{array}\right\rangle, \qquad
E_2=\left\langle \begin{array}{cccc} 0 & 0 & 1 & 0 \\ 0 & 0 & 0 & 1 \end{array}\right\rangle,
\]
\[ A=\left\langle \begin{array}{cccc} 1 & 0 & 0 & 1 \\ 0 & 1 & 1 & 0 \end{array}\right\rangle, \qquad
B=\left\langle \begin{array}{cccc} 1 & 0 & 1 & 0 \\ 0 & 1 & 1 & 1 \end{array}\right\rangle, \qquad
C=\left\langle \begin{array}{cccc} 1 & 0 & 1 & 1 \\ 0 & 1 & 0 & 1 \end{array}\right\rangle,
\]
\[ D=\left\langle \begin{array}{cccc} 1 & 0 & 1 & 0 \\ 0 & 1 & 0 & 1 \end{array}\right\rangle, \qquad
F=\left\langle \begin{array}{cccc} 1 & 0 & 0 & 1 \\ 0 & 1 & 1 & 1 \end{array}\right\rangle, \qquad
G=\left\langle \begin{array}{cccc} 1 & 0 & 1 & 1 \\ 0 & 1 & 1 & 0 \end{array}\right\rangle.
\]
Note that $\{E_1,E_2,A,B,C\}$ and $\{E_1,E_2,D,F,G\}$ form a \textbf{spread} in $E$ (a spread is a set of subspaces of the same dimension such that every $1$-dimensional space is in exactly one spread element \cite{segre}). The two spreads are isomorphic, in the sense that a change of basis of $E$ maps one to the other. Since $A$, $B$, and $C$ all intersect $D$, $F$, and $G$, deciding that at least one of $\{A,B,C\}$ is a circuit means $\{D,F,G\}$ are all independent, and vice versa. So, without loss of generality, we have completely determined the matroid $M$ if we have found which of the three $2$-dimensional spaces $A$, $B$, and $C$ are circuits and this implies that $D,F,G$ are all independent. 

\begin{Lemma}\label{3ContainsSpread}
Every $3$-dimensional space $T$ contains an element of the spread \[\{E_1,E_2,A,B,C\}.\]
\end{Lemma}
\begin{proof}
This can be done via a counting argument and the pigeon hole principle. $T$ intersects all spread elements in dimension at least $1$, since $\dim E=4$. All $1$-dimensional subspaces of $E$ are by definition contained in exactly one spread element. There are five spread elements and seven $1$-dimensional subspaces in $T$, so there has to be a spread element that contains at least two $1$-dimensional subspaces of $T$, and hence intersects it in dimension $2$. But that means the whole spread element is contained in $T$.
\end{proof}

If $A$, $B$, and $C$ are all circuits, there are no other circuits because of Lemma \ref{CircuitIntersection} and axiom (C2). If not all of $A$, $B$, and $C$ are circuits, there have to be circuits of dimension $3$. These will be all the $3$-dimensional spaces that do not contain a circuit of dimension $2$. These circuits do, however, contain an element of the spread, by Lemma \ref{3ContainsSpread}. \\

We check the circuit axioms for this construction. (C1) and (C2) are clear. For (C3), notice that the sum of every pair of circuits is equal to $E$. Thus it is sufficient to show that every $3$-space contains a circuit. This is true by construction: a 3-space either contains a $2$-dimensional circuit, or it is a circuit itself. \\

We have seen that no matter what we decide for the independence of $A$, $B$, and $C$, we always get a $q$-matroid. This means that the properties of the direct sum as in Definition \ref{def-directsum1} are not enough to determine the direct sum completely: we can make a $q$-matroid with $2$, $3$, $4$ or $5$ circuits that all satisfy this definition.

\subsection{A small non-representable $q$-matroid}
 
As a byproduct of the example in the previous section, we find a non-representable $q$-matroid in dimension $4$. The existence of non-representable $q$-(poly)matroids was established and discussed in \cite{gluesing2021qpolyindep}. However, the example here is not included in their construction and it is also the smallest possible non-representable $q$-matroid. In the classical case, the smallest non-representable matroid is of size $8$ and rank $4$ (the V\'amos matroid). For $q$-matroids it is smaller: dimension $4$ and rank $2$.

\begin{Proposition}\label{MatrixAssoc}
Let $M$ be a representable $q$-matroid over $\mathbb{F}_2$ of rank $2$ and dimension $4$, with (at least) two circuits of dimension $2$ and no loops. Then the matrix representing $M$ has the shape
$$G:=\left[\begin{array}{cccc}
1 & \alpha & 0 & 0 \\
0 & 0 & 1 & \beta
\end{array}\right],$$
with $\alpha,\beta \in \mathbb{F}_{2^m}\setminus\mathbb{F}_2$, $m>1$.
\end{Proposition}
\begin{proof}
Since $M$ has rank $2$ and dimension $4$, the shape of the matrix is 
\[ G:=\left[\begin{array}{cccc}
x_1 & x_2 & x_3 & x_4 \\
y_1 & y_2 & y_3 & y_4 
\end{array}\right], \]
with all entries in $\mathbb{F}_{2^m}$.
Without loss of generality we apply row reduction and get $x_1=1, y_1=0$.
Since there are no loops, the columns of $G$ cannot be all zero.
Consider now the two circuits. They are, without loss of generality, $E_1:=\langle(1,0,0,0),(0,1,0,0) \rangle$ and 
$E_2:=\langle(0,0,1,0),(0,0,0,1) \rangle$.
We have for $E_1$
$$\left[\begin{array}{cccc}
1 & x_2 & x_3 & x_4 \\
0 & y_2 & y_3 & y_4 
\end{array}\right] \cdot \left[\begin{array}{cc}
1 & 0 \\
0 & 1\\
0 & 0 \\
0 & 0
\end{array}\right] = \left[\begin{array}{cc}
1 & x_2 \\
0 & y_2
\end{array}\right],$$
whose rank must be one, leading to $y_2=0$.
Similarly, for $E_2$ we have
$$\left[\begin{array}{cccc}
1 & x_2 & x_3 & x_4 \\
0 & y_2 & y_3 & y_4 
\end{array}\right] \cdot \left[\begin{array}{cc}
0 & 0 \\
0 & 0\\
1 & 0 \\
0 & 1
\end{array}\right] = \left[\begin{array}{cc}
x_3 & x_4 \\
y_3 & y_4
\end{array}\right],$$
whose rank must be one, leading to the fact that $(x_3,x_4)$ and $(y_3,y_4)$ are scalar multiples. By row reduction we can conclude that $x_3=x_4=0$ and the absence of loops implies that $x_2,y_3,y_4\neq 0$. We can finally set, again by row reduction, $y_3=1$.
Note that column operations over the ground field $\mathbb{F}_2$ give an isomorphic $q$-matroid, so we have that $x_2$ and $y_4$ are elements of $\mathbb{F}_{2^m}$ but not of $\mathbb{F}_2$.
\end{proof}

\begin{Theorem}\label{NonReprDim4}
If the $q$-matroid from Section \ref{subsec-NotUnique} is representable, it cannot have $4$ circuits of dimension $2$. This gives an example of a non-representable $q$-matroid.
\end{Theorem}

\begin{proof}
We know the representation is of the form
\[ \left[\begin{array}{cccc}
1 & \alpha & 0 & 0 \\
0 & 0 & 1 & \beta
\end{array}\right], \]
with $\alpha,\beta \in \mathbb{F}_{2^m}\setminus\mathbb{F}_2$ by Proposition \ref{MatrixAssoc}.
Consider 
\[ \left[\begin{array}{cccc}
1 & \alpha & 0 & 0 \\
0 & 0 & 1 & \beta
\end{array}\right] \cdot 
\left[\begin{array}{cc}
a_0 & b_0  \\
a_1 & b_1 \\ 
a_2 & b_2 \\ 
a_3 & b_3 
\end{array}\right] =
\left[\begin{array}{cc}
a_0+a_1\alpha & b_0+b_1\alpha  \\
a_2+a_3\beta & b_2+b_3\beta
\end{array}\right]. 
\]
In order to have a circuit of dimension $2$, the determinant of this $2\times 2$ matrix should be zero. In particular, we need to have proportional columns.
This automatically tells us that $a_0=a_1=0$ implies $b_0=b_1=0$, and that $a_2=a_3=0$ implies $b_2=b_3=0$.
These two cases correspond to the two circuits $E_1$ and $E_2$ from Proposition \ref{MatrixAssoc}. Using the representations from Section \ref{subsec-NotUnique}, we found the determinants of all $2$-dimensional spaces $A,B,C,D,F,G$. They are the following:
\begin{itemize}
    \item[$A$:] $\alpha \beta+1$
    \item[$B$:] $\alpha+\beta +1$
    \item[$C$:] $\alpha \beta+\beta+ \alpha$
    \item[$D$:] $\alpha+\beta$
    \item[$F$:] $\alpha \beta+\beta+1$
     \item[$G$:] $\alpha\beta +\alpha+1$
\end{itemize}

Now, it is easy to see that if $A$ and $B$ vanish, then $C$ vanishes as well, and the same goes for $D$, $F$ and $G$. We already saw that circuits appear either in $\{A,B,C\}$ or $\{D,E,F\}$ and the other spaces are independent.
Therefore, the alternatives we have are:
\begin{itemize}
    \item none of the six determinants above vanishes, so $E_1$ and $E_2$ are the only circuits of dimension $2$;
    \item one determinant vanishes, so we have three circuits of dimension $2$;
    \item the determinants of all the elements in a spread vanish, leading to five circuits of dimension $2$ (that are all circuits in the $q$-matroid). \qedhere
\end{itemize}
\end{proof}

\begin{Corollary}
The $q$-matroid over $\mathbb{F}_2$ of rank $2$ and dimension $4$ with four circuits, as described in Theorem \ref{NonReprDim4}, is the smallest non-representable $q$-matroid.
\end{Corollary}
\begin{proof}
See the appendix for a list of all $q$-matroids with a ground space of dimension at most $3$. All of these are representable. Hence, the $q$-matroid from Theorem \ref{NonReprDim4} is the smallest non-representable $q$-matroid.
\end{proof}

\begin{example}\label{SomeExamplesRepresentables}
Consider the finite field $\FF_{64}$ and a primitive element $\alpha$ such that $\alpha^6=\alpha^4+\alpha^3+\alpha+1$. We give some examples of $q$-matroids of dimension $4$ and rank $2$, arising from our construction in Section \ref{subsec-NotUnique}, distinguishing them by the number of their circuits:
\begin{itemize}
    \item $\left[\begin{array}{cccc}
1 & \alpha^2 & 0 & 0 \\
0 & 0 & 1 & \alpha^7
\end{array}\right]$ represents a $q$-matroid with two $2$-dimensional circuits;
    \item $\left[\begin{array}{cccc}
1 & \alpha & 0 & 0 \\
0 & 0 & 1 & \alpha^8
\end{array}\right]$ represents a $q$-matroid with three $2$-dimensional circuits;
    \item $\left[\begin{array}{cccc}
1 & \alpha^{42} & 0 & 0 \\
0 & 0 & 1 & \alpha^{21}
\end{array}\right]$ represents a $q$-matroid with five $2$-dimensional circuits.
\end{itemize}
\end{example}

\begin{Remark}
Example \ref{SomeExamplesRepresentables} above also tells us something about the direct sum of two representable $q$-matroids. Suppose we have two representable $q$-matroids $M_1$ and $M_2$ over the same field $K$. Suppose these $q$-matroids are representable by matrices $G_1$ and $G_2$ over an extension field $L$ of $K$. One would expect the direct sum $M_1\oplus M_2$ to be representable by
\[ G=\left[\begin{array}{cc} G_1 & 0 \\ 0 & G_2 \end{array}\right]. \]
However, this construction is not uniquely defined, in the sense that it depends on the representations of $M_1$ and $M_2$. Over $\mathbb{F}_{64}$, we can represent the $q$-matroid $U_{1,2}$ as $\left[\begin{array}{cc} 1 & \beta \end{array}\right]$ for any $\beta\in\mathbb{F}_{64}\backslash\mathbb{F}_2$. Then the $q$-matroids of Example \ref{SomeExamplesRepresentables} are all of the form of the matrix $G$ above, so we would expect all of them to represent $U_{1,2}\oplus U_{1,2}$. However, these are not isomorphic $q$-matroids.
\end{Remark}

\section{Submodular functions and associated $q$-matroids}

Our goal is to define the direct sum of $q$-matroids in terms of matroid union. Before we can define that, we need some background on integer-valued increasing submodular functions. A function $f$ on the subspaces of $E$ is submodular if the following hold for all $A,B\subseteq E$:
\[ f(A+B)+f(A\cap B) \leq f(A)+f(B). \]
Such function can be viewed as the rank function of a $q$-polymatroid, and we refer to \cite{gluesing2021qpolyindep} for an extension of, and some overlap with, the results presented here.

The following proposition and corollary are the $q$-analogues of Proposition 11.1.1 and Corollary 11.1.2 in \cite{oxley}.

\begin{Proposition}\label{CircuitsOfMf}
Let $f$ be an integer-valued increasing submodular function on the subspaces of a finite-dimensional vector space $E$. Let
\[ \mathcal{C}(f)=\{C\subseteq E :C\text{ is  non-trivial and minimal w.r.t. inclusion  s.t. }f(C)<\dim(C)\}. \]
Then $\mathcal{C}(f)$ is the collection of circuits of a $q$-matroid $M(f)=(E,\mathcal{C}(f))$.
\end{Proposition}
\begin{proof}
We prove that $\mathcal{C}(f)$ satisfies the circuit axioms from Definition \ref{circuit-axioms}. The axiom (C1) holds by definition and, by minimality, we have (C2).\\
Let us now prove (C3). Let $C_1 \neq C_2$ be two elements of $\mathcal{C}(f)$ and let $X$ be a codimension $1$ space containing neither $C_1$ nor $C_2$ (otherwise the assertion holds void).
We have $C_i \cap X \subsetneq C_i$ for $i=1,2$. Therefore, 
$C_i \cap X \notin \mathcal{C}(f)$ by (C2) and so $\dim(C_i \cap X) \leq f(C_i \cap X)$ for $i=1,2$. Since $f$ is increasing,
\[
\dim(C_i \cap X) \leq f(C_i \cap X) \leq f(C_i) < \dim(C_i)
\]
and $\dim(C_i\cap X)=\dim(C_i)-1$, we have $\dim(C_i)-1 =f(C_i)$.

Since $f$ is increasing, is suffices to show $f((C_1+C_2)\cap X)<\dim((C_1+C_2)\cap X)$, because then $(C_1+C_2)\cap X$ contains a circuit. Now $f$ is increasing and submodular, so
\[
f((C_1+C_2)\cap X)\leq f(C_1+C_2)
\leq f(C_1)+f(C_2)-f(C_1 \cap C_2)\]
and because $C_1 \cap C_2 \subsetneq C_i$ for $i=1,2$, by minimality, $f(C_1 \cap C_2) \geq \dim(C_1 \cap C_2)$. Finally,
\begin{align*}
f((C_1+C_2)\cap X) & \leq f(C_1)+f(C_2)-f(C_1 \cap C_2) \\
 & = \dim(C_1)+\dim(C_2)-2-f(C_1 \cap C_2) \\
 & \leq \dim(C_1+C_2)-2 \\
 & =\dim((C_1+C_2)\cap X)-1.
\end{align*}
This shows that $M(f)$ is a $q$-matroid defined by its circuits $\mathcal{C}(f)$.
\end{proof}

The following is a direct result of the definition of $\mathcal{C}(f)$ and the fact that every proper subspace of a circuit is independent.

\begin{Corollary}\label{IndepInMf}
A subspace $I\subseteq E$ is independent in $M(f)$ if and only if $\dim(I')\leq f(I')$ for all nontrivial subspaces $I'$ of $I$.
\end{Corollary}

The next theorem is the $q$-analogue of \cite[Chapter 8.1 Theorem 2]{welsh1976matroid}. We point out that Theorem \ref{thm-FunctionToMatroid} and Propsition \ref{IndepOfPoly} ware already proven in \cite[Theorem 3.9]{gluesing2021qpolyindep}, but with the minimum taken over the subspaces of $A$ instead of all spaces in $E$. See also Remark \ref{r-incl-only}.

\begin{Theorem}\label{thm-FunctionToMatroid}
Let $f$ be a non-negative integer-valued increasing submodular function on the subspaces of $E$ with $f(0)=0$. Then
\[ r(A)=\min_{X\subseteq E}\{f(X)+\dim(A)-\dim(A\cap X)\} \]
is the rank function of a $q$-matroid.
\end{Theorem}
\begin{proof}
We will prove that the function $r$ satisfies the rank axioms. It is clear that $r$ is integer-valued. It is non-negative because both $f(A)$ and $\dim(A)-\dim(A\cap X)$ are non-negative. By taking $X=\{0\}$ in the definition, we get $f(\{0\})+\dim(A)-\dim(\{0\})=\dim(A)$ and therefore $r(A)\leq\dim(A)$. This proves (r1). \\
In order to prove (r2), let $A\subseteq B\subseteq E$. Then for any $X\subseteq E$, we have that $\dim(B)-\dim(A)\geq\dim(B\cap X)-\dim(A\cap X)$. It follows that
\[ f(X)+\dim(A)-\dim(A\cap X) \leq f(X)+\dim(B)-\dim(B\cap X) \]
for all $X\subseteq E$ and thus $r(A)\leq r(B)$. The proof of (r3) is rather technical, but essentially a lot of rewriting. We first claim that
\begin{align*}
& \dim(A)-\dim(A\cap X)+\dim(B)-\dim(B\cap Y) \geq \\
& \quad \dim(A+B)-\dim((A+B)\cap(X+Y))+\dim(A\cap B)-\dim((A\cap B)\cap(X\cap Y)).
\end{align*}
This statement will be used later on in the proof. By using that
\[ \dim(A)+\dim(B)=\dim(A+B)+\dim(A\cap B)\]
and multiplying by $-1$ we can rewrite our claim as
\[ \dim(A\cap X)+\dim(B\cap Y) \leq  \dim((A+B)\cap(X+Y))+\dim((A\cap B)\cap(X\cap Y)). \]
Using the modular equality again, we get
\begin{align*}
\dim((A\cap B)\cap(X\cap Y)) &= \dim((A\cap X)\cap(B\cap Y)) \\
& = \dim(A\cap X)+\dim(B\cap Y)-\dim((A\cap X)+(B\cap Y))
\end{align*}
and thus our claim is equivalent to
\[ \dim((A\cap X)+(B\cap Y))\leq\dim((A+B)\cap(X+Y)).\]
To prove this, it is enough to show the inclusion of vector spaces $(A\cap X)+(B\cap A)\subseteq(A+B)\cap(X+Y)$. Let $\mathbf{a}\in A\cap X$ and $\mathbf{b}\in B\cap Y$ be two nonzero vectors. Then $\mathbf{a}+\mathbf{b}\in(A\cap X)+(B\cap Y)$. We prove that $\mathbf{a}+\mathbf{b}\in(A+B)\cap(X+Y)$. Because $\mathbf{a}\in A$ also $\mathbf{a}\in A+B$ and because $\mathbf{a}\in X$ also $\mathbf{a}\in X+Y$. So $\mathbf{a}\in(A+B)\cap(X+Y)$. By a similar reasoning, $\mathbf{b}\in(A+B)\cap(X+Y)$. So $\mathbf{a}+\mathbf{b}\in(A+B)\cap(X+Y)$ as was to be shown. This finishes the proof of our claim. \\
We can now get back to proving axiom (r3). In the third step we use the claim together with the submodularity of $f$. In the fourth step we set $U=X+Y$ and $V=X\cap Y$. This will not produce all possible $U,V\subseteq E$, so the minimum is at least as big as the minimum over all $U,V\subseteq E$.
\begin{align*}
\lefteqn{r(A)+r(B)} \\
 & = \min_{X\subseteq E}\{f(X)+\dim(A)-\dim(A\cap X)\} + \min_{Y\subseteq E}\{f(Y)+\dim(B)-\dim(B\cap Y)\} \\
 & = \min_{X,Y\subseteq E}\{f(X)+f(Y)+\dim(A)-\dim(A\cap X)+\dim(B)-\dim(B\cap Y)\} \\
 & \geq \min_{X,Y\subseteq E}\{ f(X+Y)+f(X\cap Y)+\dim(A+B)-\dim((A+B)\cap(X+Y)) \\
 & \qquad +\dim(A\cap B)-\dim((A\cap B)\cap(X\cap Y))\} \\
 & \geq \min_{U,V\subseteq E}\{f(U)+f(V)+\dim(A+B)-\dim((A+B)\cap U) \\
 & \qquad +\dim(A\cap B)-\dim((A\cap B)\cap V)\} \\
 & = \min_{U\subseteq E}\{f(U)+\dim(A+B)-\dim((A+B)\cap U)\} \\
 & \qquad + \min_{V\subseteq E}\{f(V)+\dim(A\cap B)-\dim((A\cap B)\cap V)\} \\
 & = r(A+B)+r(A\cap B).
\end{align*}
So the rank function $r$ satisfies all rank axioms (r1),(r2),(r3).
\end{proof}

\begin{Remark}\label{r-incl-only}
Note that the minimum in Theorem \ref{thm-FunctionToMatroid} is taken over all subspaces of $E$. This is convenient for some of the proofs, but not strictly necessary. Let $X\subseteq E$ and let $X'=A\cap X$. Then 
\[ f(X')+\dim(A)-\dim(A\cap X') \leq f(X)+\dim(A)-\dim(A\cap X) \]
because $f(X')\leq f(X)$ and $\dim(A\cap X')=\dim(A\cap X)=\dim(X')$. This means that the minimum over all subspaces $X\subseteq E$ is the same as the minimum taken only over the subspaces $X'\subseteq A$. This makes calculating the rank function a lot faster in practice.
\end{Remark}

The next proposition shows that the $q$-matroids from Corollary \ref{IndepInMf} and Theorem \ref{thm-FunctionToMatroid} are the same.

\begin{Proposition}\label{IndepOfPoly}
Let $f$ be a non-negative integer-valued increasing submodular function with $f(0)=0$. Let $M(f)$ be the corresponding $q$-matroid as defined in Corollary \ref{IndepInMf} with independent spaces $\mathcal{I}$. Let $r$ be the rank function as defined in Theorem \ref{thm-FunctionToMatroid}. Then both give the same $q$-matroid because $r(I)=\dim(I)$ for all $I\in\mathcal{I}$.
\end{Proposition}
\begin{proof}
We have to prove that $r(I)=\dim(I)$ iff $\dim(I')\leq f(I')$ for all nontrivial subspace $I'$ of $I$. Note that since $f(0)=0$, this holds for all subspaces $I'$ of $I$, also the trivial one. (Note that Proposition \ref{CircuitsOfMf} does not require $f(0)=0$, but Theorem \ref{thm-FunctionToMatroid} does.)
From the remark before we have that
\[ r(I)=\min_{I'\subseteq I}\{f(I')+\dim(I)-\dim(I')\}. \]
As already proven in Theorem \ref{thm-FunctionToMatroid}, $r(I)\leq\dim(I)$. For the other inequality, the following are equivalent:
\begin{align*}
I\in\mathcal{I}(M(f)) &\Leftrightarrow f(I')\geq\dim(I') \text{ for all }I'\subseteq I \\
 &\Leftrightarrow f(I')+\dim(I)-\dim(I')\geq\dim(I) \text{ for all }I'\subseteq I \\
 &\Leftrightarrow r(I)\geq\dim(I)
\end{align*}
This proves that $r(I)=\dim(I)$.
\end{proof}

\section{Matroid union}\label{MatUn}

In this section we define the $q$-analogue of matroid union by means of its rank function and we show what are the independent spaces.

\begin{Definition}\label{DefUnione}
Let $M_1$ and $M_2$ be two $q$-matroids on the same ground space $E$, with rank functions $r_1$ and $r_2$, respectively. Then the \textbf{matroid union} $M_1\vee M_2$ is defined by the rank function
\[ r(A)=\min_{X\subseteq A}\{r_1(X)+r_2(X)+\dim A-\dim X\} \]
\end{Definition}

\begin{Theorem}\label{Union_q-Matr}
Let $M_1$ and $M_2$ be two $q$-matroids on the same ground space $E$, with rank functions $r_1$ and $r_2$, respectively. Then the matroid union $M_1\vee M_2$ is a $q$-matroid. 
\end{Theorem}
\begin{proof}
For all $A\subseteq E$, define a function $f(A)=r_1(A)+r_2(A)$. We claim that $f$ is a non-negative integer-valued submodular function on the subspaces of $E$ with $f(0)=0$. \\
Note that $r_1$ and $r_2$ are non-negative integer valued submodular functions on the subspaces of $E$ with $r_1(\{0\})=r_2(\{0\})=0$. It follows directly that $f$ is a non-negative integer-valued function on the subspaces of $E$. It is increasing, because for all $A\subseteq B\subseteq E$ we have
\[ f(A)= r_1(A)+r_2(A)\leq r_1(B)+r_2(B)=f(B). \]
Furthermore, $f$ is submodular, because for all $A,B\subseteq E$ we have
\begin{align*}
f(A+B)+f(A\cap B) & = r_1(A+B)+r_2(A+B)+r_1(A\cap B)+r_2(A\cap B) \\
& \leq r_1(A)+r_1(B)+r_2(A)+r_2(B) \\
& = f(A)+f(B).
\end{align*}
Now we apply Theorem \ref{thm-FunctionToMatroid} and Remark \ref{r-incl-only} to the function $f$: this shows that the function $r$ of Definition \ref{DefUnione} is indeed the rank function of a $q$-matroid $(E,r)$.
\end{proof}

We gather some important properties of the matroid union.

\begin{Remark}\label{dependsOnCoordinates}
The matroid union is not always invariant under coordinatisation. That is, if $\varphi:\mathcal{L}(E)\longrightarrow\mathcal{L}(E)$ is a lattice isomorphism, then it is direct from the definition that $\varphi(M_1)\vee\varphi(M_2)=\varphi(M_1\vee M_2)$. However, $M_1\vee M_2$ is not necessarily isomorphic to $\varphi(M_1)\vee M_2$. We illustrate this with a small example. \\
Let $M_1$ and $M_2$ both be isomorphic to the mixed diamond, see \ref{qcd2}. That is: $\dim(E)=2$, $r(E)=1$ and $r(A)=1$ for all $1$-dimensional spaces except one loop. Suppose the loop is at the same coordinates for both $M_1$ and $M_2$, call this subspace $\ell$. Then the rank of $M_1\vee M_2$ is one, as we will show. Consider all $X\subseteq E$. If $\dim(X)=0$ or  $\dim(X)=2$ then the expression inside the minimum of Definition \ref{DefUnione} is equal to $2$. If $\dim(X)=1$ we have to distinguish between $\ell$ and the any other space. If $X=\ell$ the expression is $0+0+2-1=1$, otherwise it is $1+1+2-1=3$. Therefore, $r(E)=1$. \\
Consider now the case where the loop of $M_1$ is $\ell_1$ and the loop of $M_2$ is $\ell_2$, with $\ell_1\neq\ell_2$. Then the calculations are as before for $\dim(X)=0$ or $\dim(X)=2$. For $\dim(X)=1$ and $X\neq\ell_1,\ell_2$ we get $1+1+2-1=3$. If $X=\ell_1$ we get $0+1+2-1=2$, and similarly for $X=\ell_2$ we get $1+0+2-1=2$. So $r(E)=2$. \\
This example illustrates that we have to be careful to define $M_1$ and $M_2$ precisely, not just up to isomorphism.
\end{Remark}

We prove two straightforward lemmas concerning the matroid union.

\begin{Lemma}\label{AddGreenDoesNothing}
Let $M_1$ and $M_2$ be two $q$-matroids on the same ground space $E$ and let $r(M_2)=0$. Then $M_1\vee M_2=M_1$ and in particular, $M\vee U_{0,n}=M$.
\end{Lemma}
\begin{proof}
The rank function of the matroid union is equal to
\[ r(A)=\min_{X\subseteq A}\{r_1(X)+0+\dim A-\dim X\}. \]
By local semimodularity (Lemma \ref{localSemimodularity}), $r_1(X)-\dim(X)+\dim(A)\geq r_1(A)$ for all $X\subseteq A$ and equality is attained for $X=A$. Hence $r(A)=r_1(A)$ and $M_1\vee M_2=M_1$.
\end{proof}

\begin{Lemma}\label{IndipUnionFromM1M2}
Let $M_1$ and $M_2$ be $q$-matroids on the same ground space $E$. Let $I$ be independent in both $M_1$ and $M_2$. Then $I$ is independent in $M_1\vee M_2$. 
\end{Lemma}
\begin{proof}
We have that $r(M_1;I)=r(M_2;I)=\dim I$ by definition. Also, all subspaces of $I$ are independent. This means that
\begin{align*}
r(M_1\vee M_2;I) &= \min_{X\subseteq I}\{r(M_1;X)+r(M_2;X)+\dim I-\dim X\} \\
&= \min_{X\subseteq I}\{\dim X+\dim X+\dim I-\dim X\} \\
&= \dim I.
\end{align*}
We conclude that $I$ is independent in $M_1\vee M_2$.
\end{proof}

The independent spaces of the matroid union can be found in the following way.

\begin{Theorem}\label{indipUnion}
Let $M_1=(E,\mathcal{I}_1)$ and $M_2=(E,\mathcal{I}_2)$ be two $q$-matroids defined by their independent spaces. Then $I\subseteq E$ is an independent space of the matroid union $M_1\vee M_2$ if and only if for all $J\subseteq I$ there exist $I_1\in\mathcal{I}_1$ and $I_2\in\mathcal{I}_2$ such that $J=I_1\oplus I_2$. We notate the collection of independent spaces of $M_1\vee M_2$ by $\mathcal{I}$.
\end{Theorem}
\begin{proof}
Let $f(A)=r_1(A)+r_2(A)$, as in the proof of Theorem \ref{Union_q-Matr}. According to Corollary \ref{IndepInMf} and Proposition \ref{IndepOfPoly}, we know that the independent spaces of the matroid union are exactly those $I\subseteq E$ such that for all nontrivial subspaces $J\subseteq I$ we have $\dim J\leq f(J)$. \\
First, let $I\subseteq E$ such that all nontrivial subspace $J\subseteq I$ can be written as $J=I_1\oplus I_2$ with $I_1\in\mathcal{I}_1$ and $I_2\in\mathcal{I}_2$. We need to prove that $I$ is independent in $M_1\vee M_2$, that is, for all $J\subseteq I$ it holds that $\dim J\leq f(J)$. This follows from
\[ \dim(J)=\dim(I_1)+\dim(I_2) = r_1(I_1)+r_2(I_2) \leq r_1(J)+r_2(J)=f(J), \] the inequality coming from the axiom (r2). \\
For the other implication, let $I$ be independent in $M_1\vee M_2$. We need to show that we can write all $I$ as $I=I_1\oplus I_2$ with $I_1\in\mathcal{I}_1$ and $I_2\in\mathcal{I}_2$. Because all subspaces of an independent space are independent, this proves the statement. \\
First, note that if $I$ is independent in $M_1\vee M_2$, then its rank is equal to its dimension: $\dim I=r(I)=\min_{X\subseteq I}\{r_1(X)+r_2(X)+\dim I-\dim X\}$. Therefore, for each $X \subseteq I$ it holds
\[ r_1(X)+r_2(X)-\dim X\geq 0. \]
We will proceed by mathematical induction on the dimension of $I$. If $\dim I=0$ then $I=\{0\}$ and we can write $\{0\}=\{0\}\oplus\{0\}$ where $\{0\}\in\mathcal{I}_1$ and $\{0\}\in\mathcal{I}_2$. If $\dim I=1$, then $r_1(I)+r_2(I)\geq1$ so $I$ is independent in at least one of $M_1$ and $M_2$. Without loss of generality, let $I$ be independent in $M_1$, then we can write $I=I\oplus\{0\}$ with $I\in\mathcal{I}_1$ and $\{0\}\in\mathcal{I}_2$. \\
Let $I$ be independent in $M_1\vee M_2$ with $\dim I=h+1$. Let $J\subseteq I$ with $\dim J=h$ and $J=J_1\oplus J_2$ for some $J_1\in\mathcal{I}_1$ and $J_2\in\mathcal{I}_2$. We will show that there is a $1$-dimensional $x\subseteq I-J$ such that either $J_1\oplus x\in\mathcal{I}_1$ or $J_2\oplus x\in\mathcal{I}_2$. \\
Assume that for all $x\subseteq I-J$ the space $J_1\oplus x$ is dependent in $M_1$. This implies that $r_1(J_1\oplus x)=r_1(J_1)$ and by \cite[Prop. 7]{JP18} we have $r_1(I)=r(J_1)=\dim J_1$. Since $I\subseteq I$ we have the following equivalent statements:
\begin{align*}
r_1(I)+r_2(I)-\dim I &\geq 0 \\
\dim J_1+r_2(I)-(\dim J_1+\dim J_2+1) &\geq 0 \\
r_2(I) &\geq \dim J_2+1
\end{align*}
and hence $r_2(I)\geq 1$. This implies that not all $x\subseteq I-J$ can be loops in $M_2$, because if they were, by semimodularity this would imply $r_2(I)=0$. So assume $x\subseteq I-J$ with $r_2(x)=1$. Then, by applying semimodularity in $M_2$ again, we get
\begin{align*}
r_2(J_2+x)+r_2(J_2\cap x) &\geq r_2(J_2)+r_2(x) \\
r_2(J_2\oplus x)+0 &\geq \dim J_2+1
\end{align*}
and it follows that $J_2\oplus x$ is independent in $M_2$. This gives the decomposition $I=J_1\oplus (J_2\oplus x)$ with $J_1\in\mathcal{I}_1$ and $J_2\oplus x\in\mathcal{I}_2$.
\end{proof}

\begin{Remark}
We want to point out that Theorem \ref{indipUnion} is indeed a $q$-analogue of the classical case. There the independent sets of the matroid union are defined by
\[ \mathcal{I}=\{I_1\cup I_2:I_1\in\mathcal{I}_1,I_2\in\mathcal{I}_2\}. \]
First of all, note that the union can be rewritten as a disjoint union. Let $I=J_1\cup J_2$ with $J_1\in\mathcal{I}_1$ and $J_2\in\mathcal{I}_2$. Take $I_1=J_1$ and $I_2=J_2-J_1$, then $I=I_1\sqcup I_2$. This procedure does not create a unique $I_1$ and $I_2$, there is a lot of choice involved. However, it does imply that every independent set $I$ of the matroid union is of the form $I=I_1\sqcup I_2$, and conversely, every $I=I_1\sqcup I_2$ is independent in the matroid union.

In the classical case, if $I=I_1\sqcup I_2$ then for all $\bar{J}\subseteq I$ we can write directly $\bar{J}=\bar{J}_1\sqcup \bar{J}_2$ with $\bar{J}_1=\bar{J}\cap I_1$ and $\bar{J}_2=\bar{J}\cap I_2$. Since $\bar{J}_1\subseteq I_1$ and $\bar{J}_2\subseteq I_2$, these are independent. This reasoning does not hold in the $q$-analogue (see also the Introduction), which is why we specifically have to state it in the definition. For a counterexample, see the example in Remark \ref{dependsOnCoordinates}: if $\ell_1=\ell_2=\ell$ we can write $E=I_1\oplus I_2$ for some $1$-dimensional $I_1$ and $I_2$ that are not equal to $\ell$, but we cannot write $\ell$ as the direct sum of independent spaces of $M_1$ and $M_2$.
\end{Remark}

\section{Matroid intersection and duality} \label{IntDual}

We complete our study of the matroid union for $q$-matroids by defining the dual operation, that is matroid intersection. We follow \cite[p.123]{welsh1976matroid}.

\begin{Definition}
Let $M_1$ and $M_2$ be $q$-matroids on the same ground space $E$ with collection of spanning spaces $\mathcal{S}(M_1)$ and $\mathcal{S}(M_2)$. Define the \textbf{$q$-matroid intersection} of $M_1$ and $M_2$ by its spanning spaces:
\[ \kS(M_1\wedge M_2)=\{S_1\cap S_2: S_1\in\kS(M_1), S_2\in\kS(M_2)\}. \]
\end{Definition}

We need to prove that it is a $q$-matroid. This can be done by checking the axioms for spanning spaces, but we can also do this by proving a more general result:
\begin{Theorem}\label{thm-intersection}
Let $M_1$ and $M_2$ be $q$-matroids on the same ground space $E$. Then
\[ M_1\wedge M_2=(M_1^*\vee M_2^*)^*. \]
\end{Theorem}
\begin{proof}
From \cite{BCR21} we know that the orthogonal complements of the spanning spaces of a $q$-matroid $M$ are the independent spaces of the dual $q$-matroid $M^*$. So we have to prove that the orthogonal complements of $\kS(M_1\wedge M_2)$ are the independent spaces of $M_1^*\vee M_2^*$. \\
First, start with $S\in\kS(M_1\wedge M_2)$. Then we can write $S=S_1\cap S_2$ for $S_1\in\kS(M_1)$ and $S_2\in\kS(M_2)$. Let $S'\supseteq S$ be a superspace of $S$. Then we can write $S'=T_1\cap T_2$ for $T_1\supseteq S_1$ and $T_2\supseteq S_2$. Note that $T_1$ and $T_2$ are also spanning spaces of $M_1$ and $M_2$, respectively. Now we take orthogonal complements.

The orthogonal complement is $S^\perp=S_1^\perp+S_2^\perp$. Now we can write $S_1^\perp=I_1^*$ with $I_1^*$ independent in $M_1^*$, and similarly, $S_2^\perp=I_2^*$ is independent in $M_2^*$. We need to prove that $I_1^*+I_2^*$ is independent in $M_1^*\vee M_2^*$, that is, we have to show that all $I'\subseteq I_1^*+I_2^*$ can be written as $J_1\oplus J_2$ with $J_1$ independent in $M_1^*$ and $J_2$ independent in $M_2^*$. Note that all $I'\subseteq I_1^*+I_2^*$ can be written as $I'=(S')^\perp$, with $S'=T_1\cap T_2$ as above. If we take orthogonal complements of $T_1$ and $T_2$, we get independent spaces of $M_1^*$ and $M_2^*$. So we can write $I'=J_1\oplus J_2$. (We can always make the sum a direct sum by taking a subspace of $J_1$ if necessary.) We conclude that $I_1^*+I_2^*$ is independent in $M_1^*\vee M_2^*$.

For the opposite inclusion, start with an independent space $I$ of $M_1^*\vee M_2^*$. Then by Theorem \ref{indipUnion} we can write $I=I_1+I_2$ with $I_1$ independent in $M_1^*$ and $I_2$ independent in $M_2^*$. Taking orthogonal complements gives that $I^\perp=I_1^\perp\cap I_2^\perp=S_1\cap S_2$ for spanning spaces $S_1$ in $M_1$ and $S_2$ in $M_2$. This implies that $I^\perp$ is in $\kS(M_1\wedge M_2)$.

Since $(M_1^*\vee M_2^*)^*$ is a $q$-matroid, this shows that $M_1\wedge M_2$ is a $q$-matroid as well.
\end{proof}

We have the following corollary on intersection, union, and restriction and contraction.

\begin{Corollary}\label{DualUnionIntersection}
Let $M_1$ and $M_2$ be $q$-matroids on the same ground space $E$. Then, for $T\subseteq E$,
\[ (M_1\vee M_2)|_T = M_1|_T\vee M_2|_T \]
and also
\[ (M_1\wedge M_2)/T \cong (M_1/T)\wedge(M_2/T). \]
\end{Corollary}
\begin{proof}
The first part of the statement follows directly from Definition \ref{DefUnione} of the matroid union and the fact that for the rank function of the restriction is $r_{M|_T}(A)=r(A)$. The second statement follows from the first by applying Theorem \ref{thm-intersection}, use $(M/T)^* \cong M^*|_{T^\perp}$ from Theorem \ref{DualRestrContr}, and then applying Theorem \ref{thm-intersection} again.
\begin{align*}
(M_1/T)\wedge(M_2/T) & = ((M_1/T)^*\vee(M_2/T)^*)^* \\
 & \cong ((M_1^*|_{T^\perp})\vee(M_2^*|_{T^\perp}))^* \\
 & = ((M_1^*\vee M_2^*)|_{T^\perp})^* \\
 & \cong (M_1^*\vee M_2^*)^*/T \\
 & = (M_1\wedge M_2)/T. \qedhere
\end{align*}
\end{proof}

We finish this section with the dual of Lemma \ref{AddGreenDoesNothing}.

\begin{Lemma}\label{IntersectAllRed}
Let $M$ be a $q$-matroid. Then $M= M\wedge U_{n,n}$.
\end{Lemma}
\begin{proof}
Applying Lemma \ref{AddGreenDoesNothing} to $M^*$ gives that $M^*= M^*\vee U_{0,n}=(M\wedge U_{n,n})^*$. Dualising both sides gives the desired result.
\end{proof}

\section{The direct sum}\label{DirSum}

In this section we will define the direct sum of two $q$-matroids. The idea will be to first add loops to $M_1$ and $M_2$, so they are on the same ground space, and then taking their matroid union. In the classical case, we can also write the direct sum like this: the idea for this construction comes from \cite[Proposition 7.6.13 part 2]{brylawski_1986}.

\subsection{Defining the direct sum} \label{AdL}

The next definition explains how to ``add a loop'' to a $q$-matroid.

\begin{Definition}\label{AggiungiLoop}
Let $M=(E,r)$ be a $q$-matroid. Then the direct sum of $M$ and a loop $\ell$ is denoted by $M'=M\oplus \ell$ and constructed in the following way. Let $E'=E+\ell$. Then for every $A'\subseteq E'$ we can write $A'+\ell=A\oplus\ell$ for a unique $A\subseteq E$. Then $r'(A')=r(A)$.
\end{Definition}

\begin{Remark}\label{caseSplitting}
The definition above divides the subspaces of $E'$ into three different kinds.
\begin{itemize}
\item If $A'\subseteq E$ then $A'=A$ and $\dim A'=\dim A$.
\item If $A'\supseteq \ell$ then $A=A'\cap E$ and $\dim A=\dim A'-1$.
\item If $A'$ is not contained in $E$ and does not contain $\ell$, then $\dim A'=\dim A$. There is a diamond with bottom $A'\cap A\subseteq E$, top $A'+\ell$ and with $A$ and $A'$ in between.
\end{itemize}
\end{Remark}

This construction is well defined, in the sense that it gives a $q$-matroid, as the next theorem shows.

\begin{Theorem}\label{AddLoop}
The direct sum $M'=M\oplus\ell$ as defined above is a $q$-matroid, that is, the rank function $r'$ satisfies (r1),(r2),(r3).
\end{Theorem}
\begin{proof}
(r1) Since $r(A)\geq0$ we have $r'(A')\geq0$ as well. We get that $r'(A')=r(A)\leq\dim A\leq\dim A'$ by Remark \ref{caseSplitting}.

(r2) 
Let $A'\subseteq B'$. Since
$A=(A'+l)\cap E$,
$B=(B'+l)\cap E$ and 
$A'+l \subseteq B'+l$, we have that $A \subseteq B$. Therefore $r'(A')=r(A)\leq r(B)=r'(B')$.

For (r3) let $A',B'\subseteq E'$. We first claim that $(A'+B')+\ell=(A+B)+\ell$ and $(A'\cap B')+\ell=(A\cap B)+\ell$, because this implies that 
\begin{align*}
r'(A'+B') + r'(A'\cap B') & = r(A+B) + r(A\cap B) \\
& \leq r(A)+r(B) \\
& = r'(A')+r'(B').
\end{align*}
Now let us prove the claims. For addition, we see that
\[ (A'+B')+\ell=(A'+\ell)+(B'+\ell)=(A+\ell)+(B+\ell)=(A+B)+\ell. \]
For intersection we distinguish three cases depending on whether $A'$ and $B'$ contain $\ell$.
\begin{itemize}
\item Let $\ell\not\subseteq A',B'$. Then $(A'\cap B')+\ell=(A'+\ell)\cap(B'+\ell)=(A+\ell)\cap(B+\ell)=(A\cap B)+\ell$.
\item Let $\ell\subseteq A',B'$. Then $(A'\cap B')+\ell=A'\cap B'=(A+\ell)\cap(B+\ell)=(A\cap B)+\ell$.
\item Let $\ell\subseteq A'$ and $\ell\not\subseteq B'$. Then $(A'\cap B')+\ell=((A'\cap E)\cap B')+\ell=((A'\cap E)+\ell)\cap(B'+\ell)=(A+\ell)\cap(B+\ell)=(A\cap B)+\ell$.
\end{itemize}
The function $r'$ satisfies the axioms (r1),(r2),(r3), hence $M$ is a $q$-matroid.
\end{proof}

We combine the adding of loops and the matroid union to define the direct sum. 
\begin{Definition}\label{DirSumWithUnion}
Let $M_1=(E_1,r_1)$ and $M_2=(E_2,r_2)$ be two $q$-matroids on trivially intersecting ground spaces. Let $n_1=\dim E_1$ and $n_2=\dim E_2$. We construct the direct sum $M_1\oplus M_2$ as follows.
\begin{itemize}
\item Let $E=E_1\oplus E_2$. This will be the ground space of $M$. By slight abuse of notation, we denote by $E_i$ both the ground space of $M_i$ and the embedding of $E_i$ in $E$.
\item In the lattice $\mathcal{L}(E)$ we have that the intervals $[0,E_1]$ and $[E_2,1]$ are isomorphic to $\mathcal{L}(E_1)$, and the intervals $[0,E_2]$ and $[E_1,1]$ are isomorphic to $\mathcal{L}(E_2)$. Fix the involution $\perp$ such that $E_1^\perp=E_2$.
\item Add $n_2$ times a loop to $M_1$, using Theorem \ref{AddLoop}. This gives the $q$-matroid $M_1'$ on ground space $E$. Assume that $M_1'|_{E_1}\cong M_1$ and $M_1'|_{E_2}\cong U_{0,n_2}$.
\item Add $n_1$ times a loop to $M_2$, using again Theorem \ref{AddLoop}. This gives the $q$-matroid $M_2'$ on ground space $E$. Assume that $M_2'|_{E_1}\cong U_{0,n_1}$ and $M_2'|_{E_2}\cong M_2$.
\end{itemize}
Now the direct sum is defined as $M_1\oplus M_2=M_1'\vee M_2'$, with the matroid union as in Theorem \ref{Union_q-Matr}.
\end{Definition}

Note that this procedure is well-defined, since we already showed that adding loops and taking the matroid union are well-defined constructions. We do, however, have to show that this procedure always defines the same $q$-matroid up to isomorphism, since it was observed in Remark \ref{dependsOnCoordinates} that matroid union is not invariant under coordinatisation.

\begin{Theorem}\label{IsomorphicIsLinearAlgebra}
Let $M_1=(E_1,r_1)$ and $M_2=(E_2,r_2)$ be two $q$-matroids on trivially intersecting ground spaces and let $M=M_1\oplus M_2$ be their direct sum as constructed in Definition \ref{DirSumWithUnion}. Let $\varphi_i$ be a lattice-isomorphism of $\mathcal{L}(E_i)$ for $i=1,2$. Then there is an isomorphism $\psi$ of $\mathcal{L}(E)$ such that $\varphi_1(M_1)\oplus\varphi_2(M_2)=\psi(M)$.
\end{Theorem}
\begin{proof}
Let $\psi$ be an isomorphism on $\mathcal{L}(E)$ such that $\psi|_{E_1}=\varphi_1$ and $\psi|_{E_2}=\varphi_2$. We can construct $\psi$ by its images of $n_1+n_2$ linearly independent $1$-dimensional spaces: we find these by taking the image under $\varphi_1$ of $n_1$ linearly independent $1$-spaces in $E_1$ and the image under $\varphi_2$ of $n_2$ linearly independent $1$-spaces in $E_2$. \\
Let $A\subseteq E$ and let $B\subseteq E_1$ such that $A+E_2=B\oplus E_2$. This means that $B=(A+E_2)\cap E_1$. Now we have that
\begin{align*}
(\psi(A)+E_2)\cap E_1 &= (\psi(A)+\psi(E_2))\cap E_1 \\
&= \psi(A+E_2)\cap E_1 \\
&= \psi(B\oplus E_2) \cap E_1 \\
&= (\psi(B)\oplus\psi(E_2))\cap E_1 \\
&= (\varphi_1(B)\oplus E_2)\cap E_1 \\
&= \varphi_1(B) \\
&= \varphi_1((A+E_2)\cap E_1).
\end{align*}
The rank function of $\varphi_1(M_1)'$ is equal to
\begin{align*}
r(\varphi_1(M_1)';A) & = r(\varphi_1(M_1);(A+E_2)\cap E_1) \\
& = r(M_1;\varphi_1((A+E_2)\cap E_1)) \\
& = r(M_1; (\psi(A)+E_2)\cap E_1) \\
& = r(M_1';\psi(A) ).
\end{align*}
We have a similar argument for $M_2$ and $\varphi_2$. Combining these gives that
\begin{align*}
\lefteqn{r(\varphi_1(M_1)\oplus\varphi_2(M_2);A)} \\
 & = r(\varphi_1(M_1)'\vee\varphi_2(M_2)';A) \\
 &= \min_{X\subseteq E}\{ r(\varphi_1(M_1)';X)+r(\varphi_2(M_2)';X)+\dim A-\dim(A\cap X)\} \\
&= \min_{\psi(X)\subseteq E}\{ r(M_1';\psi(X))+r(M_2';\psi(X))+\dim\psi(A)-\dim(\psi(A)\cap \psi(X))\} \\
&= r(M;\psi(A)) = r(\psi(M);A).
\end{align*}
This proves the theorem.
\end{proof}

The next lemma is a direct consequence of Theorem \ref{thm:directsumrank}, but we prove it now to make the calculations in the next section easier.

\begin{Lemma}\label{l-RankOfSum}
For two $q$-matroids $M_1$ and $M_2$ it holds that 
\[r(M_1\oplus M_2)=r(M_1)+r(M_2).\]
\end{Lemma}
\begin{proof}
By applying Definitions \ref{AggiungiLoop} and \ref{DirSumWithUnion}, we get that
\[ r(M_1\oplus M_2)=r(M_1'\vee M_2';E)=\min_{X\subseteq E}\{ r(M_1';X)+r(M_2';X)+\dim E-\dim X\}. \]
If we take $X=E$, we get that
\[ r(M_1';X)+r(M_2';X)+\dim E-\dim X = r(M_1';E)+r(M_2';E)=r(M_1)+r(M_2). \]
Now let $Y_1\subseteq E_1$ such that $X+E_2=Y_1\oplus E_2$. Then $r(M_1';X)=r(M_1;Y_1)$. Similarly, let $Y_2\subseteq E_2$ such that $X+E_1=Y_2\oplus E_1$ so $r(M_2';X)=r(M_2;Y_2)$. We have that $\dim(Y_1)=\dim(X)-\dim(X\cap E_2)$ and $\dim(Y_2)=\dim(X)-\dim(X\cap E_1)$. Note that, by local semimodularity (Lemma \ref{localSemimodularity}, $r(M_1;Y_1)\geq r(M_1;E_1)-\dim(E_1)+\dim(Y_1)$ and similarly $r(M_2;Y_2)\geq r(M_2;E_2)-\dim(E_2)+\dim(Y_2)$. All together this gives
\begin{align*}
\lefteqn{r(M_1';X)+r(M_2';X)+\dim E-\dim X} \\
 &= r(M_1;Y_1)+r(M_2;Y_2)+\dim E-\dim X \\
 &\geq r(M_1;E_1)-\dim(E_1)+\dim(Y_1) \\
 & \quad +r(M_2;E_2)-\dim(E_2)+\dim(Y_2)+\dim E-\dim X \\
 &= r(M_1)+r(M_2)-\dim(X)+\dim(Y_1)+\dim(Y_2) \\
 &= r(M_1)+r(M_2)-\dim(X)+\dim(X)-\dim(X\cap E_2)+\dim(X)-\dim(X\cap E_1) \\
 &= r(M_1)+r(M_2)+\dim(X)-\dim(X\cap E_2)-\dim(X\cap E_1) \\
 &\geq r(M_1)+r(M_2).
\end{align*}
This means that the minimum $\min_{X\subseteq E}\{ r(M_1';X)+r(M_2';X)+\dim E-\dim X\}$ is attained by $X=E$ and $r(M_1\oplus M_2)=r(M_1)+r(M_2)$.
\end{proof}

\subsection{Examples of the direct sum}

To get some feeling for this construction, we analyse some small examples. We refer to the Appendix for an overview of small $q$-matroids. \\
We start with the easiest examples possible, with $n_1=n_2=1$.

\begin{Example}
Let $M_1=M_2=U_{0,1}$. This is the sum of two loops. In fact, we could just use Theorem \ref{AddLoop} here, without Definition \ref{DirSumWithUnion}, but we do the whole procedure for clarity. For $M_1'=M_1\oplus\ell$, let $E_1=\langle(1,0)\rangle$. Then by Theorem \ref{AddLoop}, $M_1'$ is a $q$-matroid of rank $0$, so all its subspaces have rank zero. In fact, $M_1'\cong U_{0,2}$. Let $E_2=\langle(0,1)\rangle$. We also have that $M_2'\cong U_{0,2}$. Applying Theorem \ref{Union_q-Matr} we find that $M_1'\vee M_2'\cong U_{0,2}$. \\
Let $M_1=U_{0,1}$ and $M_2=U_{1,1}$. Then $M_1'=U_{0,2}$ as argued above. For $M_2'$, let $E_2=\langle(0,1)\rangle$ and apply Theorem \ref{AddLoop}. By construction, $r(\{0\})=0$. In dimension $1$ we have $r(\langle (0,1)\rangle)=r(E_2)=r_2(E_2)=1$, $r(\ell)=r_2(\{0\})=0$, and for all other spaces $A$ of dimension $1$ we have $r(A)=r_2(E_2)=1$. These are the three cases in Remark \ref{caseSplitting}. Note that $M_2'$ is a mixed diamond (see Section \ref{qcd2}). Finally, we have $r(E)=r_2(E_2)=1$. By Lemma \ref{AddGreenDoesNothing}, $M_1 \oplus M_2=M_1' \vee M_2' =M_2'$. \\
The last case to consider is $M_1=M_2=U_{1,1}$. We have seen that $M_1'$ and $M_2'$ are a mixed diamond. To get $M_1'\vee M_2'$, we first see that $r(\{0\})=0$. In dimension $1$, we have that $r(\langle(0,1)\rangle)=\min\{ 0+0+1-0, 1+0+1-1 \}=1$. For $r(\langle(1,0)\rangle)$ we get the same but in a different order, so the rank is again $1$. For a $1$-dimensional space not equal to $E_1$ or $E_2$ we get $r(A)=\min\{0+0+1-0, 1+1+1-1 \}=1$. Finally, for $E$ we get $r(E)=\min\{0+0+2-0,1+0+2-1,0+1+2-1,1+1+2-1,1+1+2-0\}=2$. So, $M_1\oplus M_2=U_{2,2}$.
\end{Example}

Note that it follows from this example that $U_{1,2}$ is connected: it cannot be written as a direct sum of two $q$-matroids of dimension $1$.

\begin{Example}\label{ex-PrimePlusRed}
We calculate the $q$-matroid $P_1^*$ (see Section \ref{P1star}), it is the sum of a prime diamond (see Section \ref{qcd2}) and an independent 1-dimensional space, that is, $M_1=U_{1,2}$ and $M_2=U_{1,1}$. Let $E_1=\langle (0,0,1),(0,1,0)\rangle$ and $E_2=\langle (1,0,0)\rangle$. We first have to make $M_1'$ and $M_2'$. \\
For $M_1'$ we take $\ell=E_2=\langle(1,0,0)\rangle$. We have that $r_1'(0)=0$ and $r_1'(E)=r(M_1)=1$. For a $1$-dimensional space inside $E_1$, the rank is $1$, while $r_1'(\ell)=0$. For any other $1$-dimensional space $A$, $r_1'(A)=r_1(A')$ for $A'\subseteq E_1$, so $r_1'(A)=1$. For the $2$-dimensional spaces $A$, $r_1'(E_1)=1$. If $\ell\subseteq A$, $r_1'(A)=r_1(A\cap E_1)=1$. For the other $2$-dimensional spaces we have $r_1'(A)=r_1(E_1)=1$. Together, we find that $M_1'$ is the $q$-matroid $P_1$ in the Section \ref{P1}. \\
For $M_2'$ we have to add a loop twice to $U_{1,1}$. The first loop gives the mixed diamond, as explained in the previous example. The second one gives a $q$-matroid isomorphic to $P_2$ (see Section \ref{P2}). \\
Now we take the union. We have $r(0)=0$ and also $r(E)=2$ by Lemma \ref{l-RankOfSum}. \\ There are three types of $1$-dimensional spaces, as well as three types of $2$-dimensional spaces. Let $\dim A=1$. If $A\subseteq E_1$ then $r(A)=\min\{0+0+1-0, 1+0+1-1\}=1$. \\If $A=E_2$ then $r(A)=\min\{0+0+1-0,0+1+1-1\}=1$. For the other $1$-dimensional spaces $A$, $r(A)=\min\{0+0+1-0,1+1+1-1\}=1$. Now let $\dim A=2$. If $A=E_1$ then $r(A)=\min\{0+0+2-0,1+0+2-1,1+0+2-2\}=1$. For the other 2-dimensional spaces $A$, note that any $1$-dimensional space has rank $1$ in either $M_1'$ or in $M_2'$, contributing $1+0+2-1=0+1+2-1=2$ to the minimum. The zero space also contributes $0+0+2-0=2$, and the space itself gives $1+1+2-2=2$. So $r(A)=2$. \\
In total, we see that $U_{1,2}\oplus U_{1,1}\cong P_1^*$.
\end{Example}

\subsection{Properties of the direct sum}

We will now show that the direct sum as defined here has some desirable properties. All of these results are also true for the classical case, motivating the `correctness' of the definition of the direct sum presented in the previous section. Further support of the definition is provided by \cite{GJ}, where it is shown that the direct sum is the coproduct in the category of $q$-matroids and linear weak maps.

\begin{Theorem}\label{thm:directsumrank}
Let $M_1$ and $M_2$ be two $q$-matroids with ground spaces $E_1$ and $E_2$, respectively. Let their direct sum be as defined in Definition \ref{DirSumWithUnion}. Then for any $A\subseteq E$ of the form $A=A_1\oplus A_2$ with $A_1\subseteq E_1$ and $A_2\subseteq E_2$ it holds that $r(M_1\oplus M_2;A)=r(M_1;A_1)+r(M_2;A_2)$.
\end{Theorem}
\begin{proof}
By definition of the direct sum we have that
\[ r(M_1\oplus M_2;A)=\min_{X\subseteq A}\{r(M_1';X)+r(M_2';X)+\dim A-\dim X\}. \]
We will show that the minimum is attained for $X=A$. First, note that $A+E_2=A_1\oplus E_2$ and $A+E_1=A_2\oplus E_1$. Then taking $X=A$ inside the minimum gives
\[ r(M_1';A)+r(M_2';A)+\dim A-\dim A=r(M_1;A_1)+r(M_2;A_2). \]
We have left to show that for any $X\subseteq A$, the quantity inside the minimum is at least $r(M_1;A_1)+r(M_2;A_2)$. To see this, take $B_1\subseteq E_1$ and $B_2\subseteq E_2$ such that $X+E_2=B_1\oplus E_2$ and $X+E_1=B_2\oplus E_1$. \\
For the dimension of $B_1$, we have that $\dim B_1=\dim(X+E_2)-\dim E_2=\dim X-\dim(X\cap E_2)$. Furthermore, $B_1\subseteq A_1$ and thus by local semimodularity (Lemma \ref{localSemimodularity}), $r(M_1;A_1)-\dim A_1\leq r(M_1;B_1)-\dim B_1$. Similar results hold for $B_2$. Finally, note that $\dim B_1+\dim B_2\leq\dim X$. \\
Combining this, we get that
\begin{align*}
\lefteqn{r(M_1';X)+r(M_2';X)+\dim A-\dim X} \\
&= r(M_1;B_1)+r(M_2;B_2)+\dim A-\dim X \\
 &\geq r(M_1;A_1)-\dim A_1+\dim B_1+r(M_2;A_2)-\dim A_2+\dim B_2+\dim A-\dim X \\
&\geq r(M_1;A_1)+r(M_2;A_2).
\end{align*}
This completes the proof that $r(M_1\oplus M_2;A)=r(M_1;A_1)+r(M_2;A_2)$.
\end{proof}

From Theorem \ref{RangoSomma} the following is now immediate.

\begin{Corollary}\label{MinorTheorem}
Let $M_1$ and $M_2$ be two $q$-matroids with ground spaces $E_1$ and $E_2$, respectively. Then their direct sum, as defined in Definition \ref{DirSumWithUnion}, satisfies the properties of Definition \ref{def-directsum1}.
\end{Corollary}

Note that this implies that also the rest of the results in Section \ref{FirstDef} hold for our Definition \ref{DirSumWithUnion} of the direct sum. Another desirable property of our definition of the direct sum is that the dual of the direct sum is the direct sum of the duals. \\

In order to prove that direct sum commutes with duality, we need to define duality on $E_1$, $E_2$, and $E$ in a compatible way.

\begin{Definition}
Let $E=E_1\oplus E_2$ and let $\perp$ be an anti-isomorphism on $\mathcal{L}(E)$ such that $E_1^\perp=E_2$. Define an anti-isomorphism $\perp\!\!(E_1)$ on $E_1$ by
\[ A^{\perp(E_1)}:=(A+E_2)^\perp=A^\perp\cap E_2^\perp=A^\perp\cap E_1. \]
Similarly, we define the anti-isomorphism $A^{\perp(E_2)}=A^\perp\cap E_2$.
\end{Definition}

The map $\perp\!\!(E_1)$ (and, similarly, $\perp\!\!(E_2)$) is indeed an anti-isomorphism, because it is the concatenation of the isomorphism $[0,E_1]\to[E_2,E]$ given by $A\mapsto A\oplus E_2$ and the anti-isomorphism $\perp$ restricted to $[E_2,E]\to[0,E_1]$. 

\begin{Theorem}\label{thm-DualDirect}
Let $M_1$ and $M_2$ be $q$-matroids on ground spaces $E_1$ and $E_2$, respectively. Then we have that $(M_1\oplus M_2)^*=M_1^*\oplus M_2^*$.
\end{Theorem}
\begin{proof}
Let $B$ be a basis of $M_1\oplus M_2$. We will prove that $B^\perp$ is a basis of $M_1^*\oplus M_2^*$. First, note that by Lemma \ref{l-RankOfSum} we have
\begin{align*}
r(M_1^*\oplus M_2^*) &= r(M_1^*)+r(M_2^*) \\
&= \dim E_1-r(M_1)+\dim E_2-r(M_2) \\
&= \dim E-r(M_1\oplus M_2) \\
&= \dim B^\perp.
\end{align*}
This means that if we show that $B^\perp$ is independent in $M_1^*\oplus M_2^*$, it is also a basis. The rank of $B^\perp$ in $M_1^*\oplus M_2^*$ is given by
\[ r(M_1^*\oplus M_2^*, B^\perp)=\min_{X\subseteq B^\perp}\{r((M_1^*)';X)+r((M_2^*)';X)+\dim B^\perp-\dim X\}. \]
We want this to be equal to $\dim B^\perp$, hence we need to show for all $X\subseteq B^\perp$ that
\[ r((M_1^*)';X)+r((M_2^*)';X) \geq \dim X. \]
This bound is tight: take $X=\{0\}$ for example. In order to rewrite the left hand side of this inequality, note that 
\begin{align*}
((X+E_2)\cap E_1)^{\perp(E_1)} & = ((X+E_2)\cap E_1)^\perp\cap E_1 \\
 &= ((X+E_2)^\perp + E_2)\cap E_1 \\
 &= ((X^\perp\cap E_1)+E_2)\cap E_1 \\
 &= X^\perp \cap E_1
\end{align*}
because for a space in $E_1$, first adding $E_2$ and then intersecting with $E_1$ is giving the same space we start with. With this in mind, we can rewrite one of the rank functions:
\begin{align*}
r((M_1^*)';X) &= r(M_1^*;(X+E_2)\cap E_1) \\
 &= r(M_1;((X+E_2)\cap E_1)^{\perp(E_1)}) +\dim((X+E_2)\cap E_1)-r(M_1;E_1) \\
 &= r(M_1;X^\perp \cap E_1) +\dim E_1-\dim(X^\perp\cap E_1)-r(M_1;E_1).
\end{align*}
We have a similar result for $r((M_2^*)';X)$. Applying this yields
\begin{align*}
\lefteqn{r((M_1^*)';X)+r((M_2^*)';X)} \\
 &= r(M_1;X^\perp \cap E_1) +\dim E_1 -\dim(X^\perp\cap E_1)-r(M_1;E_1) \\
 & \quad + r(M_2;X^\perp \cap E_2) +\dim E_2-\dim(X^\perp\cap E_2)-r(M_2;E_2) \\
 &= \dim X +\dim X^\perp -\dim B + r(M_1;X^\perp \cap E_1) -\dim(X^\perp\cap E_1) \\
 & \quad + r(M_2;X^\perp \cap E_2) -\dim(X^\perp\cap E_2).
\end{align*}
In order for this quantity to be greater than or equal to $\dim X$, we need to prove for all $X\subseteq B^\perp$ the following inequality:
\[ r(M_1;X^\perp\cap E_1) + r(M_2;X^\perp\cap E_2) + \dim X^\perp \geq \dim B+\dim(X^\perp\cap E_1)+\dim(X^\perp\cap E_2). \]
We proceed by mathematical induction on $\dim X^\perp$, so the base case is $X^\perp=B$. We claim that $r(M_1;B\cap E_1)=\dim(B\cap E_1)$. Since $B$ is a basis, it holds for all $Y\subseteq B$ that $r(M_1';Y)+r(M_2';Y) \geq \dim Y$ (by a reasoning as in the beginning of this proof). In particular, this holds for $Y=B\cap E_1\subseteq B$, so
\[ r(M_1';B\cap E_1)+r(M_2';B\cap E_1) = r(M_1;B\cap E_1) + 0 \geq \dim(B\cap E_1) \]
and thus by the rank axiom (r2) equality holds and we prove our claim. By the same reasoning, we have that $r(M_2;B\cap E_2)=\dim(B\cap E_2)$. This implies the induction step of our proof:
\[ r(M_1;B\cap E_1) + r(M_2;B\cap E_2) + \dim B = \dim B+\dim(B\cap E_1)+\dim(B\cap E_2). \]
Now assume the inequality holds for all $Y\supseteq B$ with $\dim Y\leq d$, where $\dim B\leq d\leq\dim E$. Consider a space $Y$ with $\dim Y=d+1$ and write $Y=Y'\oplus x$ for some $1$-dimensional subspace $x$. Since $x$ cannot be in both $E_1$ and $E_2$, we can assume without loss of generality that $x\not\subseteq E_1$ for any choice of $x$ such that $Y=Y'\oplus x$ (the case $x\not\subseteq E_2$ goes similarly). Then by rewriting and using the induction hypothesis we get
\begin{align*}
\lefteqn{r(M_1;(Y'\oplus x)\cap E_1) + r(M_2;(Y'\oplus x)\cap E_2) + \dim (Y'\oplus x)} \\
 &= r(M_1;Y'\cap E_1) + r(M_2;(Y'\oplus x)\cap E_2) + \dim Y' +1 \\
 &\geq r(M_1;Y'\cap E_1) + r(M_2;Y'\cap E_2) + \dim Y' +1 \\
 &\geq \dim B+\dim(Y'\cap E_1)+\dim(Y'\cap E_2) + 1\\
 &= \dim B+\dim((Y'\oplus x)\cap E_1)+\dim(Y'\cap E_2)+1 \\
 &\geq \dim B+\dim((Y'\oplus x)\cap E_1)+\dim((Y'\oplus x)\cap E_2).
\end{align*}
This concludes the proof that $B^\perp$ is independent in $M_1^*\oplus M_2^*$, hence a basis, and we have proven that $(M_1\oplus M_2)^*=M_1^*\oplus M_2^*$. 
\end{proof}

In the last example we will answer the question started in Section \ref{ExDim4} about the direct sum of two copies of $U_{1,2}$. This direct sum is now uniquely defined.

\begin{Example}\label{2U12}
Let $M_1=M_2=U_{1,2}$. We will compute $M:=M_1 \oplus M_2$. This $q$-matroid is defined as $M=U_{1,2}'\vee U_{1,2}'$. \\
Let us coordinatize the ground space of $M_1$ as 
$E_1=\langle (1,0,0,0), (0,1,0,0) \rangle$ and that of $M_2$ as $E_2=\langle (0,0,1,0), (0,0,0,1) \rangle$. Let $E=E_1\oplus E_2$. \\
We first compute $U_{1,2}'$. Since $n_1=n_2=2$, we need to add two loops to $U_{1,2}$ via Definition \ref{AggiungiLoop}. This gives a $q$-matroid with ground space $E$ and $r(A)=1$ for each $A\subseteq E$, unless $A\subseteq E_2$, then $r(A)=0$. \\
To determine $M=U_{1,2}'\vee U_{1,2}'$ we use Lemma \ref{l-RankOfSum} to get $r(M)=2$. By Proposition \ref{prop-noloops}, $M$ does not have any loops. So it suffices to decide for every $2$-dimensional space $A$ whether it is a basis or a circuit. First, note that
\[r(A)=\min_{X \subseteq A} \{r_1(X)+r_2(X)+\dim(A)-\dim(X)\}
=\min_{X \subseteq A} \{r_1(X)+r_2(X)+2-\dim(X)\}.\]
We distinguish between different types of $2$-spaces, depending on their intersection with $E_1$ and $E_2$.
\begin{itemize}
\item For $A=E_1=E_2$ we have $r(A)=1$ by Corollary \ref{MinorTheorem}.
\item Let $A\cap E_1 =A \cap E_2=\{0\}$, then
\begin{itemize}
\item if $\dim(X)=0$ then $r_1(X)+r_2(X)+2-\dim(X)=2$; 
\item if $\dim(X)=1$ then $r_1(X)+r_2(X)+2-\dim(X)=3$;
\item if $\dim(X)=2$ then $r_1(X)+r_2(X)+2-\dim(X)=2$;
\end{itemize}
so we conclude that $r(A)=2$.
\item In the case $\dim(A\cap E_2)=1$ and $A\cap E_1 =\{0\}$ (or vice versa) we have
\begin{itemize}
\item if $\dim(X)=0$ then $r_1(X)+r_2(X)+2-\dim(X)=2$; 
\item if $\dim(X)=1$ then $r_1(X)+r_2(X)+2-\dim(X)=3$ if $X$ is not contained in $E_2$, and $r_1(X)+r_2(X)+2-\dim(X)=2$ otherwise;
\item if $\dim(X)=2$ then $r_1(X)+r_2(X)+2-\dim(X)=2$;
\end{itemize}
so we conclude that $r(A)=2$.
\item Finally, if $\dim(A\cap E_2)=\dim(A\cap E_1)=1$ we have
that
\begin{itemize}
\item if $\dim(X)=0$ then $r_1(X)+r_2(X)+2-\dim(X)=2$; 
\item if $\dim(X)=1$ then $r_1(X)+r_2(X)+2-\dim(X)=3$ if $X$ is not contained in $E_1$ nor in $E_2$, and $r_1(X)+r_2(X)+2-\dim(X)=2$ otherwise;
\item if $\dim(X)=2$ then $r_1(X)+r_2(X)+2-\dim(X)=2$;
\end{itemize}
so we conclude that $r(A)=2$.
\end{itemize}
We see that all $2$-spaces except $E_1$ and $E_2$ are basis. Since we have $E_1=E_2^\perp$, it follows that this $q$-matroid is self-dual. Because $U_{1,2}^*=U_{1,2}$, this example is in agreement with Theorem \ref{thm-DualDirect}.
\end{Example}

\section{Connectedness}\label{Connect}

In the classical case, every matroid is the direct sum of its connected components. It therefore makes sense to consider the notion of connectedness in the study of the direct sum of $q$-matroids. In this final section we collect some thoughts and examples concerning a possible $q$-analogue of connectedness. We will not be able to define the concept, but we hope to argue why it is not straightforward and give some possible paths for further investigation. \\

To define connectedness in classical matroids, we use the following relation on the elements of a matroid $M=(E,r)$.
\begin{quote}
Two elements $x,y\in E$ are related if either $x=y$ or if there is a circuit of $M$ that contains both $x$ and $y$.
\end{quote}
This relation is in fact an equivalence relation \cite[Theorem 3.36]{gordonmcnulty}. We call a matroid connected if it has only one equivalence class under this relation. If there are multiple equivalence classes $E_1,\ldots,E_k$ then we can write
\[ M=M|_{E_1}\oplus\cdots\oplus M|_{E_k}. \]
We will discuss some attempts to find a $q$-analogue of this equivalence relation. Note that we are looking for an equivalence relation on the $1$-dimensional spaces of $E$.

\subsection{First attempt}

The first obvious $q$-analogue for the relation is the following:
\begin{Definition}\label{def-failed1}
Two $1$-dimensional spaces $x,y\subseteq E$ are related if either $x=y$ or if there is a circuit of $M$ that contains both $x$ and $y$.
\end{Definition}
However, this is not an equivalence relation, because it is not transitive. Look at the matroid $P_1$ from the catalogue (Section \ref{P1}). The spaces $\langle(0,1,0)\rangle$ and $\langle(0,0,1)\rangle$ are in a circuit, and also $\langle (0,0,1)\rangle$ and $\langle(1,1,0)\rangle$ are in a circuit, but $\langle(0,1,0)\rangle$ and $\langle(1,1,0)\rangle$ are not in a circuit.

\subsection{Alternative attempt}

Assume we have a $q$-matroid $M=(E,r)$ with $\mathcal{H}$ its family of hyperplanes.

\begin{Definition}\label{relation_hyperplanes}
Let $x$ and $y$ be two $1$-dimensional spaces in $E$. We say $x$ and $y$ are related if $x=y$ or if there is a hyperplane $H\in\mathcal{H}$ such that $x,y\not\subseteq H$. We call this relation $R$.
\end{Definition}

\begin{Remark}
For classical matroids, consider the following relations:
\begin{itemize}
\item $x$ and $y$ are related if $x=y$ or if there is a circuit containing both $x$ and $y$.
\item $x$ and $y$ are related if $x=y$ or if there is a hyperplane containing neither $x$ nor $y$.
\end{itemize}
It is a well established result for classical matroids (see for example \cite[Theorem 3.36]{gordonmcnulty}) that the first relation is an equivalence relation. It is also a classical result \cite[Theorem 3.48]{gordonmcnulty} that both relations give the same equivalence classes. However, the $q$-analogues of these two relations are \emph{not} equivalent. Being in a circuit is equivalent to being in the orthogonal complement of a hyperplane, not being outside a hyperplane. So the relation defined in this subsection is not equivalent to the relation in the previous subsection. In fact, Definition \ref{relation_hyperplanes} is an equivalence relation, as the next theorem shows.
\end{Remark}

\begin{Theorem}
The relation $R$ from Definition \ref{relation_hyperplanes} is an equivalence relation.
\end{Theorem}
\begin{proof}
We follow the proof of \cite[Proposition 3.36]{gordonmcnulty}, replacing circuits with hyperplanes and reversing inclusion. $R$ is clearly reflexive and symmetric. So we only have to prove it is transitive. We will frequently use the following hyperplane axiom \cite{BCR21}:
\begin{itemize}
\item[(H3')] If $H_1,H_2\in\mathcal{H}$ with $y\not\subseteq H_1,H_2$ and $x\subseteq H_2$, $x\not\subseteq H_1$, then there is an $H_3\in\mathcal{H}$ such that $(H_1\cap H_2)+y\subseteq H_3$ and $x\not\subseteq H_3$.
\end{itemize}
Let $x,y,z$ be $1$-dimensional spaces in $E$. Let $x,y\not\subseteq H_1$ and $y,z\not\subseteq H_2$. We have to show there exists a hyperplane $H'$ not containing $x$ and $z$. If $x\not\subseteq H_2$ or $z\not\subseteq H_1$, we are done, so suppose $x\subseteq H_2$ and $z\subseteq H_1$. We will use induction on $\dim H_1-\dim(H_1\cap H_2)$. \\
Suppose $\dim H_1-\dim(H_1\cap H_2)=1$, then we can write $H_1$ as $(H_1\cap H_2)+z$. Applying (H3') yields an $H'\in\mathcal{H}$ such that $(H_1\cap H_2)+y\subseteq H'$ and $x\not\subseteq H'$. We need to have that $z\not\subseteq H'$, because otherwise $H_1\subsetneq H'$ and this violates axiom (H2). So $H'$ is a hyperplane not containing $x$ and $z$, as requested. \\
Now suppose $\dim H_1-\dim(H_1\cap H_2)=n>1$ and assume that $H'$ exists for all pairs of hyperplanes such that $\dim H_1-\dim(H_1\cap H_2)<n$. We will use (H3') twice to find a hyperplane $H_4\in\mathcal{H}$ such that $\dim H_1-\dim(H_1\cap H_4)<\dim H_1-\dim(H_1\cap H_2)$ and such that $x\subseteq H_4$, $x\not\subseteq H_1$ and $z\subseteq H_1$, $z\not\subseteq H_4$. Then we can apply the induction hypothesis to $H_1$ and $H_4$.
\[ \includegraphics[width=.8\textwidth]{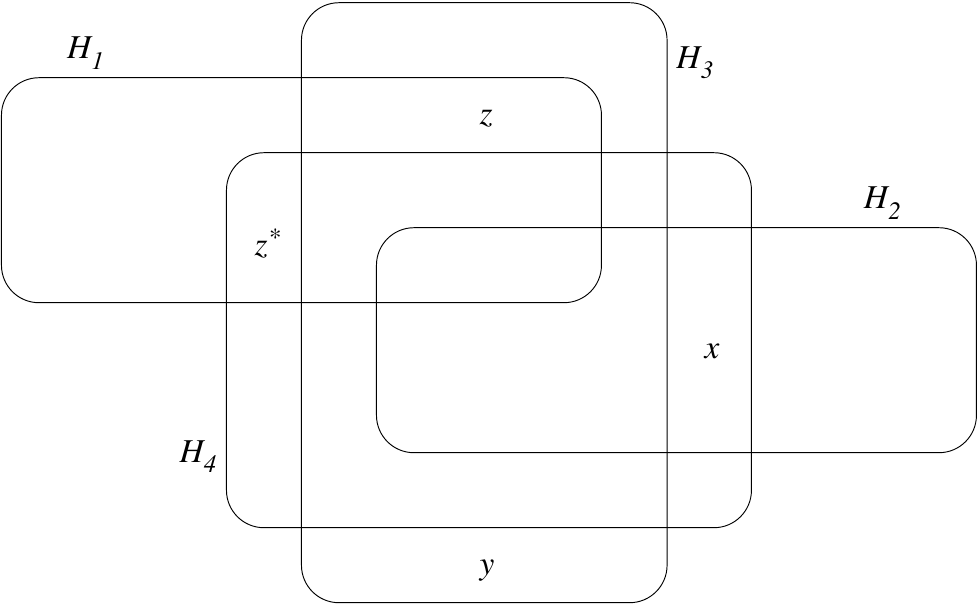} \]
First we apply (H3') to $H_1$ and $H_2$. This gives $H_3\in\mathcal{H}$ such that $(H_1\cap H_2)+y\subseteq H_3$ and $x\not\subseteq H_3$. If $z\not\subseteq H_3$ we are done, so let $z\subseteq H_3$. However, there is a $1$-dimensional space $z^*\subseteq H_1$, $z^*\not\subseteq H_2$ such that $z^*\not\subseteq H_3$: if not, $H_1\subsetneq H_3$ and this violates axiom (H2). \\
Now we apply (H3') again, to $H_2$ and $H_3$ with $z^*\not\subseteq H_2,H_3$ and $z\subseteq H_3$, $z\not\subseteq H_2$. This gives $H_4\in\mathcal{H}$ such that $(H_2\cap H_3)+z^*\subseteq H_4$ and $z\not\subseteq H_4$. If $x\not\subseteq H_4$ we are done, so let $x\subseteq H_4$. \\
By construction (see picture) we have that $(H_1\cap H_2)\subseteq(H_1\cap H_4)$. This inclusion is strict, because $z^*\subseteq H_1,H_4$ but $z^*\not\subseteq H_2$. This means we have $\dim H_1-\dim(H_1\cap H_4)<\dim H_1-\dim(H_1\cap H_4)$. By the induction hypothesis, we can now find an $H'\in\mathcal{H}$ such that $x,z\not\subseteq H'$. \\
This proves that the relation $R$ is transitive, and hence an equivalence relation.
\end{proof}

The good news is that we have found a relation that is in fact an equivalence relation. The bad news is that it does not work like we want to. The uniform $q$-matroids $U_{0,3}$ and $U_{3,3}$ only have one equivalence class, where we would want that $U_{0,3}$ is the sum of three copies of $U_{0,1}$ and $U_{3,3}$ is the sum of three copies of $U_{1,1}$. Also the $q$-matroid $P_1^*$ (Section \ref{P1star}) in the catalog has only one equivalence class, where we constructed it in Example \ref{ex-PrimePlusRed} as the direct sum $U_{1,1}\oplus U_{1,2}$. $P_1$ on the other hand (the dual of $P_1^*$) has more than one equivalence class: a signal that this attempt for an equivalence relation does not play nice with duality.

\subsection{Towards a well-defined definition}

As we saw, Definition \ref{def-failed1} is in general not an equivalence relation. However, in some $q$-matroids it is an equivalence relation. From our examples, we think the following statements could be true.

\begin{Conjecture}
The relation of Definition \ref{def-failed1} is an equivalence relation in at least one of $M$ and $M^*$.
\end{Conjecture}

\begin{Conjecture}
Let $M$ be a $q$-matroid with circuits $\mathcal{C}$ and cocircuits $\mathcal{C}^*$. Suppose $\dim(C\cap C^*)\neq1$ for all $C\in\mathcal{C}$ and $C^*\in\mathcal{C}^*$. Then Definition \ref{def-failed1} is an equivalence relation.
\end{Conjecture}

Both conjectures are of course true in the classical case. To see this for the last conjecture, note that it can be proven that the intersection between a circuit and a cocircuit can never be a single element. See for example \cite[Proposition 2.1.11]{oxley}. The $q$-analogue of this statement is not true in general: see for example the $q$-matroid $P_1^*$ of Section \ref{P1star}. It has one circuit, $\langle(0,1,0),(0,0,1)\rangle$, that intersects in dimension $1$ with the cocircuit $\langle(1,1,0),(0,0,1)\rangle$. \\

We welcome any further hints towards a better understanding of the $q$-analogues of the direct sum, connectedness, and their relation.

\section*{Acknowledgement}

We would like to express our gratitude to Heide Gluesing-Luerssen and Benjamin Jany for carefully reading the paper. Their questions and remarks have lead to significant corrections and clarifications in the manuscript. We thank also the reviewers of this manuscript for their valuable comments.

\bibliographystyle{abbrv}
\bibliography{bibliography} 

\appendix
\section{A catalogue of small $q$-matroids}\label{qcatalogue}

In this appendix we make a list of all $q$-matroids with a ground space up to dimension $3$. We hope that these explicit examples help the reader developing intuition on $q$-matroids. We represent the $q$-matroids as a colouring of the underlying subspace lattice: if a cover is red, the rank goes up; if a cover is green, the rank stays the same. See \cite{BCJ17} for more information on matroidal bicolourings.
When defining a space as the span of some vectors, for space reasons, we remove parentheses and commas. As an example, the space generated by $(0,1,0)$ will be denoted by $\langle 010 \rangle$.

\subsection{Dimension 0}\label{qcd0}

There is only one $q$-matroid of dimension zero: the uniform $q$-matroid $U_{0,0}$. This is independent of the field over which the ground space is defined.

\subsection{Dimension 1}\label{qcd1}

There are two $q$-matroids of dimension one: the uniform $q$-matroids $U_{0,1}$ and $U_{1,1}$. This is independent of the field over which the ground space is defined. Their representations are $[0]$ and $[1]$.

\subsection{Dimension 2}\label{qcd2}

It is proven in \cite{BCJ17} that a bicolouring is matroidial if and only if it is one of the following four options:
\[ \begin{array}{cc}
\hspace{-40pt} \xygraph{
[] *+{\dir{o}}
	(-@[red][ddlll(.5)] *+{\dir{o}}
		(-@[red][ddrrr(.5)] *+{\dir{o}}="0"
		)
	,-@[red][dl] *+{\dir{o}}
		(-@[red]"0"
		)
	,[d] *[]{\ldots}
	,-@[red][dr] *+{\dir{o}}
		(-@[red]"0"
		)
	,-@[red][ddrrr(.5)] *+{\dir{o}}
		(-@[red]"0"
		)
	)
} \hspace{-40pt} &
\hspace{-40pt} \xygraph{
[] *+{\dir{o}}
	(-@[green][ddlll(.5)] *+{\dir{o}}
		(-@[green][ddrrr(.5)] *+{\dir{o}}="0"
		)
	,-@[green][dl] *+{\dir{o}}
		(-@[green]"0"
		)
	,[d] *[]{\ldots}
	,-@[green][dr] *+{\dir{o}}
		(-@[green]"0"
		)
	,-@[green][ddrrr(.5)] *+{\dir{o}}
		(-@[green]"0"
		)
	)
} \hspace{-40pt} \\[-24pt]
\text{one} & \text{zero} \\[12pt]
\hspace{-40pt} \xygraph{
[] *+{\dir{o}}
	(-@[green][ddlll(.5)] *+{\dir{o}}
		(-@[red][ddrrr(.5)] *+{\dir{o}}="0"
		)
	,-@[green][dl] *+{\dir{o}}
		(-@[red]"0"
		)
	,[d] *[]{\ldots}
	,-@[green][dr] *+{\dir{o}}
		(-@[red]"0"
		)
	,-@[red][ddrrr(.5)] *+{\dir{o}}
		(-@[green]"0"
		)
	)
} \hspace{-40pt} &
\hspace{-40pt} \xygraph{
[] *+{\dir{o}}
	(-@[green][ddlll(.5)] *+{\dir{o}}
		(-@[red][ddrrr(.5)] *+{\dir{o}}="0"
		)
	,-@[green][dl] *+{\dir{o}}
		(-@[red]"0"
		)
	,[d] *[]{\ldots}
	,-@[green][dr] *+{\dir{o}}
		(-@[red]"0"
		)
	,-@[green][ddrrr(.5)] *+{\dir{o}}
		(-@[red]"0"
		)
	)
} \hspace{-40pt} \\[-24pt]
\text{mixed} & \text{prime}
\end{array} \]

This implies that there are also four $q$-matroids of dimension $2$. The \textbf{one} diamond is the $q$-matroid $U_{2,2}$ of rank $2$, represented by the identity matrix $I_2$. There are two $q$-matroids of rank $1$ and dimension $2$: the uniform $q$-matroid $U_{1,2}$ given by the \textbf{prime} diamond, represented by $\left[\begin{array}{cc}1&\alpha \end{array}\right]$ and the $q$-matroid given by the \textbf{mixed} diamond, represented as $\left[\begin{array}{cc} 1&0 \end{array}\right]$ where $1$ and $\alpha$ are algebraically independent. Finally, the only $q$-matroid with dimension $2$ and rank $0$ is $U_{0,2}$.

\subsection{Dimension 3}\label{qcd3}

One can argue that this is the first dimension where things get interesting.

\begin{Theorem}
There are $8$ $q$-matroids of dimension $3$.
\end{Theorem}
\begin{proof}
By duality, we only need to show that there are four $q$-matroids of dimension $3$ and rank $0$ or $1$. There is, as in any dimension, one $q$-matroid of rank $0$: $U_{0,3}$. For a rank $1$ $q$-matroid it suffices to say which $1$-dimensional subspaces of the ground space are independent (i.e., bases) and dependent (i.e., loops). Since loops come in subspaces \cite[Lemma 11]{JP18}, we determine the $q$-matroid completely by picking a dimension for the loop space. Since a loopspace in a $q$-matroid of rank $1$ has dimension at most $2$ (otherwise it would have rank $0$), the loopspace can have dimension $0$, $1$ or $2$. This gives three $q$-matroids of rank $1$.
\end{proof}

We will now explicitly list all eight $q$-matroids of dimension $3$. For convenience, we do this over the field $\mathbb{F}_2$, but the general construction of the theorem above holds for other fields as well.

\subsubsection{$U_{0,3}$}

\begin{minipage}{.55\textwidth}
\resizebox{\textwidth}{!}{ \xygraph{
[] *+[]{\tiny{\txt{$E$}}}
	(-@[green][dlll] *+[]{\tiny{\txt{100 \\  010}}}
		(-@[green][d] *+[]{\tiny{\txt{ 100}}}="100"
			(-@[green][drrr] *+[]{\tiny{\txt{0}}}="0"
			)
		,-@[green][dr] *+[]{\tiny{\txt{010}}}="010"
			(-@[green]"0"
			)
		,-@[green][drr] *+[]{\tiny{\txt{110}}}="110"
			(-@[green]"0"
			)
		)
	,(-@[green][dll] *+[]{\tiny{\txt{100 \\ 011}}}
		(-@[green]"100"
		,-@[green][drr] *+[]{\tiny{\txt{111}}}="111"
			(-@[green]"0"
			)
		,-@[green][drrr] *+[]{\tiny{\txt{011}}}="011"
			(-@[green]"0"
			)
		)
	,(-@[green][dl] *+[]{\tiny{\txt{100 \\ 001}}}
		(-@[green]"100"
		,-@[green][drrr] *+[]{\tiny{\txt{001}}}="001"
			(-@[green]"0"
			)
		,-@[green][drrrr] *+[]{\tiny{\txt{101}}}="101"
			(-@[green]"0"
			)
		)
	,(-@[green][d] *+[]{\tiny{\txt{010 \\ 001}}}
		(-@[green]"010"
		,-@[green]"011"
		,-@[green]"001"
		)
	,(-@[green][dr] *+[]{\tiny{\txt{101 \\ 010}}}
		(-@[green]"010"
		,-@[green]"111"
		,-@[green]"101"
		)
	,(-@[green][drr] *+[]{\tiny{\txt{101 \\ 011}}}
		(-@[green]"110"
		,-@[green]"011"
		,-@[green]"101"
		)
	,(-@[green][drrr] *+[]{\tiny{\txt{110 \\ 001}}}
		(-@[green]"110"
		,-@[green]"111"
		,-@[green]"001"
		)
    )
} }
\end{minipage}\hfill
\begin{minipage}{.4\textwidth}
Rank: $r(E)=0$ \\
Independent: $0$\\
Bases: $0$ \\
Circuits: all $1$-spaces \\
Hyperplanes: none\\
Cocircuits: none \\
Dual: $U_{3,3}$\\
Direct sum: $U_{0,1}\oplus U_{0,1}\oplus U_{0,1}$. \\
Representation: $\left[\begin{array}{ccc} 0 & 0 & 0 \end{array}\right]$
\end{minipage}

\subsubsection{$U_{1,3}$}

\begin{minipage}{.55\textwidth}
\resizebox{\textwidth}{!}{ \xygraph{
[] *+[]{\tiny{\txt{$E$}}}
	(-@[green][dlll] *+[]{\tiny{\txt{100 \\  010}}}
		(-@[green][d] *+[]{\tiny{\txt{ 100}}}="100"
			(-@[red][drrr] *+[]{\tiny{\txt{0}}}="0"
			)
		,-@[green][dr] *+[]{\tiny{\txt{010}}}="010"
			(-@[red]"0"
			)
		,-@[green][drr] *+[]{\tiny{\txt{110}}}="110"
			(-@[red]"0"
			)
		)
	,(-@[green][dll] *+[]{\tiny{\txt{100 \\ 011}}}
		(-@[green]"100"
		,-@[green][drr] *+[]{\tiny{\txt{111}}}="111"
			(-@[red]"0"
			)
		,-@[green][drrr] *+[]{\tiny{\txt{011}}}="011"
			(-@[red]"0"
			)
		)
	,(-@[green][dl] *+[]{\tiny{\txt{100 \\ 001}}}
		(-@[green]"100"
		,-@[green][drrr] *+[]{\tiny{\txt{001}}}="001"
			(-@[red]"0"
			)
		,-@[green][drrrr] *+[]{\tiny{\txt{101}}}="101"
			(-@[red]"0"
			)
		)
	,(-@[green][d] *+[]{\tiny{\txt{010 \\ 001}}}
		(-@[green]"010"
		,-@[green]"011"
		,-@[green]"001"
		)
	,(-@[green][dr] *+[]{\tiny{\txt{101 \\ 010}}}
		(-@[green]"010"
		,-@[green]"111"
		,-@[green]"101"
		)
	,(-@[green][drr] *+[]{\tiny{\txt{101 \\ 011}}}
		(-@[green]"110"
		,-@[green]"011"
		,-@[green]"101"
		)
	,(-@[green][drrr] *+[]{\tiny{\txt{110 \\ 001}}}
		(-@[green]"110"
		,-@[green]"111"
		,-@[green]"001"
		)
    )
} }
\end{minipage}\hfill
\begin{minipage}{.4\textwidth}
Rank: $r(E)=1$ \\
Independent: $0$, all $1$-spaces\\
Bases: all 1-spaces \\
Circuits: all $2$-spaces \\
Hyperplanes: $0$\\
Cocircuits: $E$ \\
Dual: $U_{2,3}$\\
Direct sum: no. \\
Representation: $\left[\begin{array}{ccc} 1 & \alpha & \alpha^2 \end{array}\right]_{\mathbb{F}_{2^3}}$
\end{minipage}

\subsubsection{$P_1$: rank 1, $1$-dimensional loopspace} \label{P1}

\begin{minipage}{.55\textwidth}
\resizebox{\textwidth}{!}{ \xygraph{
[] *+[]{\tiny{\txt{$E$}}}
	(-@[green][dlll] *+[]{\tiny{\txt{100 \\  010}}}
		(-@[red][d] *+[]{\tiny{\txt{ 100}}}="100"
			(-@[green][drrr] *+[]{\tiny{\txt{0}}}="0"
			)
		,-@[green][dr] *+[]{\tiny{\txt{010}}}="010"
			(-@[red]"0"
			)
		,-@[green][drr] *+[]{\tiny{\txt{110}}}="110"
			(-@[red]"0"
			)
		)
	,(-@[green][dll] *+[]{\tiny{\txt{100 \\ 011}}}
		(-@[red]"100"
		,-@[green][drr] *+[]{\tiny{\txt{111}}}="111"
			(-@[red]"0"
			)
		,-@[green][drrr] *+[]{\tiny{\txt{011}}}="011"
			(-@[red]"0"
			)
		)
	,(-@[green][dl] *+[]{\tiny{\txt{100 \\ 001}}}
		(-@[red]"100"
		,-@[green][drrr] *+[]{\tiny{\txt{001}}}="001"
			(-@[red]"0"
			)
		,-@[green][drrrr] *+[]{\tiny{\txt{101}}}="101"
			(-@[red]"0"
			)
		)
	,(-@[green][d] *+[]{\tiny{\txt{010 \\ 001}}}
		(-@[green]"010"
		,-@[green]"011"
		,-@[green]"001"
		)
	,(-@[green][dr] *+[]{\tiny{\txt{101 \\ 010}}}
		(-@[green]"010"
		,-@[green]"111"
		,-@[green]"101"
		)
	,(-@[green][drr] *+[]{\tiny{\txt{101 \\ 011}}}
		(-@[green]"110"
		,-@[green]"011"
		,-@[green]"101"
		)
	,(-@[green][drrr] *+[]{\tiny{\txt{110 \\ 001}}}
		(-@[green]"110"
		,-@[green]"111"
		,-@[green]"001"
		)
    )
} }
\end{minipage}\hfill
\begin{minipage}{.4\textwidth}
Rank: $r(E)=1$ \\
Independent: $0$, all $1$-spaces except $\langle100\rangle$\\
Bases: all 1-spaces except $\langle100\rangle$ \\
Circuits: $\langle100\rangle$, $\langle010,001\rangle$, $\langle101,010\rangle$, $\langle101,011\rangle$, $\langle110,001\rangle$  \\
Hyperplanes: $\langle100\rangle$\\
Cocircuits: $\langle010,001\rangle$ \\
Dual: $M_1^*$\\
Direct sum: $U_{0,1}\oplus U_{1,2}$ \\
Representation: $\left[\begin{array}{ccc} 0 & 1 & \alpha \end{array}\right]_{\mathbb{F}_{2^2}}$
\end{minipage}

\subsubsection{$P_2$: rank 1, $2$-dimensional loopspace} \label{P2}

\begin{minipage}{.55\textwidth}
\resizebox{\textwidth}{!}{ \xygraph{
[] *+[]{\tiny{\txt{$E$}}}
	(-@[red][dlll] *+[]{\tiny{\txt{100 \\  010}}}
		(-@[green][d] *+[]{\tiny{\txt{ 100}}}="100"
			(-@[green][drrr] *+[]{\tiny{\txt{0}}}="0"
			)
		,-@[green][dr] *+[]{\tiny{\txt{010}}}="010"
			(-@[green]"0"
			)
		,-@[green][drr] *+[]{\tiny{\txt{110}}}="110"
			(-@[green]"0"
			)
		)
	,(-@[green][dll] *+[]{\tiny{\txt{100 \\ 011}}}
		(-@[red]"100"
		,-@[green][drr] *+[]{\tiny{\txt{111}}}="111"
			(-@[red]"0"
			)
		,-@[green][drrr] *+[]{\tiny{\txt{011}}}="011"
			(-@[red]"0"
			)
		)
	,(-@[green][dl] *+[]{\tiny{\txt{100 \\ 001}}}
		(-@[red]"100"
		,-@[green][drrr] *+[]{\tiny{\txt{001}}}="001"
			(-@[red]"0"
			)
		,-@[green][drrrr] *+[]{\tiny{\txt{101}}}="101"
			(-@[red]"0"
			)
		)
	,(-@[green][d] *+[]{\tiny{\txt{010 \\ 001}}}
		(-@[red]"010"
		,-@[green]"011"
		,-@[green]"001"
		)
	,(-@[green][dr] *+[]{\tiny{\txt{101 \\ 010}}}
		(-@[red]"010"
		,-@[green]"111"
		,-@[green]"101"
		)
	,(-@[green][drr] *+[]{\tiny{\txt{101 \\ 011}}}
		(-@[red]"110"
		,-@[green]"011"
		,-@[green]"101"
		)
	,(-@[green][drrr] *+[]{\tiny{\txt{110 \\ 001}}}
		(-@[red]"110"
		,-@[green]"111"
		,-@[green]"001"
		)
    )
} }
\end{minipage}\hfill
\begin{minipage}{.4\textwidth}
Rank: $r(E)=1$ \\
Independent:  $0$ all $1$-spaces except $\langle 100\rangle,\langle 010 \rangle,\langle 110 \rangle$\\
Bases: all independents except 0 \\
Circuits: the three loops  $\langle 100\rangle,\langle 010 \rangle,\langle 110 \rangle$\\
Hyperplanes: $\langle 100 , 010 \rangle$\\
Cocircuits: $\langle 001 \rangle$ \\
Dual: $M_2^*$\\
Direct sum: $U_{0,1}\oplus U_{0,1}\oplus U_{1,1}$. \\
Representation: $\left[\begin{array}{ccc} 0 & 0 & 1 \end{array}\right]$
\end{minipage}

\subsubsection{$P_2^*$} \label{P2star}

\begin{minipage}{.55\textwidth}
\resizebox{\textwidth}{!}{ \xygraph{
[] *+[]{\tiny{\txt{$E$}}}
	(-@[green][dlll] *+[]{\tiny{\txt{100 \\  010}}}
		(-@[red][d] *+[]{\tiny{\txt{ 100}}}="100"
			(-@[red][drrr] *+[]{\tiny{\txt{0}}}="0"
			)
		,-@[red][dr] *+[]{\tiny{\txt{010}}}="010"
			(-@[red]"0"
			)
		,-@[red][drr] *+[]{\tiny{\txt{110}}}="110"
			(-@[red]"0"
			)
		)
	,(-@[green][dll] *+[]{\tiny{\txt{100 \\ 011}}}
		(-@[red]"100"
		,-@[red][drr] *+[]{\tiny{\txt{111}}}="111"
			(-@[red]"0"
			)
		,-@[red][drrr] *+[]{\tiny{\txt{011}}}="011"
			(-@[red]"0"
			)
		)
	,(-@[red][dl] *+[]{\tiny{\txt{100 \\ 001}}}
		(-@[green]"100"
		,-@[red][drrr] *+[]{\tiny{\txt{001}}}="001"
			(-@[green]"0"
			)
		,-@[green][drrrr] *+[]{\tiny{\txt{101}}}="101"
			(-@[red]"0"
			)
		)
	,(-@[red][d] *+[]{\tiny{\txt{010 \\ 001}}}
		(-@[green]"010"
		,-@[green]"011"
		,-@[red]"001"
		)
	,(-@[green][dr] *+[]{\tiny{\txt{101 \\ 010}}}
		(-@[red]"010"
		,-@[red]"111"
		,-@[red]"101"
		)
	,(-@[green][drr] *+[]{\tiny{\txt{101 \\ 011}}}
		(-@[red]"110"
		,-@[red]"011"
		,-@[red]"101"
		)
	,(-@[red][drrr] *+[]{\tiny{\txt{110 \\ 001}}}
		(-@[green]"110"
		,-@[green]"111"
		,-@[red]"001"
		)
    )
} }
\end{minipage}\hfill
\begin{minipage}{.4\textwidth}
Rank: $r(E)=2$ \\
Independent:  $0$ all $1$-spaces except $\langle 001\rangle$, and as 2-spaces $\langle 100, 010\rangle$, $\langle 100, 011\rangle$, $\langle 101, 010\rangle$\\
Bases: $\langle 100, 010\rangle$, $\langle 100, 011\rangle$, $\langle 101, 010\rangle$ \\
Circuits: the loop   $\langle 001\rangle$\\
Hyperplanes: $\langle 100 , 001 \rangle$, $\langle 010, 001\rangle$,  $\langle 110, 001\rangle$ \\
Cocircuits: $\langle 010 \rangle$, $\langle 100\rangle$, $\langle 110 \rangle$ \\
Dual: $M_2^*$\\
Direct sum: $U_{1,1}\oplus U_{1,1}\oplus U_{0,1}$. \\
Representation: $\left[\begin{array}{ccc}  1 & 0&0 \\
0& 1 & 0 \end{array}\right]$
\end{minipage}

\subsubsection{$P_1^*$}\label{P1star}

\begin{minipage}{.55\textwidth}
\resizebox{\textwidth}{!}{ \xygraph{
[] *+[]{\tiny{\txt{$E$}}}
	(-@[green][dlll] *+[]{\tiny{\txt{100 \\ 010}}}
		(-@[red][d] *+[]{\tiny{\txt{ 100}}}="100"
			(-@[red][drrr] *+[]{\tiny{\txt{0}}}="0"
			)
		,-@[red][dr] *+[]{\tiny{\txt{010}}}="010"
			(-@[red]"0"
			)
		,-@[red][drr] *+[]{\tiny{\txt{110}}}="110"
			(-@[red]"0"
			)
		)
	,(-@[green][dll] *+[]{\tiny{\txt{100 \\ 011}}}
		(-@[red]"100"
		,-@[red][drr] *+[]{\tiny{\txt{111}}}="111"
			(-@[red]"0"
			)
		,-@[red][drrr] *+[]{\tiny{\txt{011}}}="011"
			(-@[red]"0"
			)
		)
	,(-@[green][dl] *+[]{\tiny{\txt{100 \\ 001}}}
		(-@[red]"100"
		,-@[red][drrr] *+[]{\tiny{\txt{001}}}="001"
			(-@[red]"0"
			)
		,-@[red][drrrr] *+[]{\tiny{\txt{101}}}="101"
			(-@[red]"0"
			)
		)
	,(-@[red][d] *+[]{\tiny{\txt{010 \\ 001}}}
		(-@[green]"010"
		,-@[green]"011"
		,-@[green]"001"
		)
	,(-@[green][dr] *+[]{\tiny{\txt{101 \\ 010}}}
		(-@[red]"010"
		,-@[red]"111"
		,-@[red]"101"
		)
	,(-@[green][drr] *+[]{\tiny{\txt{101 \\ 011}}}
		(-@[red]"110"
		,-@[red]"011"
		,-@[red]"101"
		)
	,(-@[green][drrr] *+[]{\tiny{\txt{110 \\ 001}}}
		(-@[red]"110"
		,-@[red]"111"
		,-@[red]"001"
		)
    )
} }
\end{minipage}\hfill
\begin{minipage}{.4\textwidth}
Rank: $r(E)=2$ \\
Independent: $0$, all $1$-spaces, all $2$-spaces except $\langle010,001\rangle$\\
Bases: all 2-spaces except $\langle010,001\rangle$ \\
Circuits: $\langle010,001\rangle$ \\
Hyperplanes: $\langle100\rangle$, $\langle110\rangle$, $\langle111\rangle$, $\langle101\rangle$, $\langle010,001\rangle$ \\
Cocircuits: $\langle010,001\rangle$, $\langle110,001\rangle$, $\langle101,011\rangle$, $\langle101,010\rangle$, $\langle100\rangle$ \\
Dual: $M_1$\\
Direct sum: $U_{1,1}\oplus U_{1,2}$ \\
Representation: $\left[\begin{array}{ccc} 1 & 0 & 0 \\ 0 & \alpha & 1 \end{array}\right]_{\mathbb{F}_{2^2}}$
\end{minipage}

\subsubsection{$U_{2,3}$}

\begin{minipage}{.55\textwidth}
\resizebox{\textwidth}{!}{ \xygraph{
[] *+[]{\tiny{\txt{$E$}}}
	(-@[green][dlll] *+[]{\tiny{\txt{100 \\  010}}}
		(-@[red][d] *+[]{\tiny{\txt{ 100}}}="100"
			(-@[red][drrr] *+[]{\tiny{\txt{0}}}="0"
			)
		,-@[red][dr] *+[]{\tiny{\txt{010}}}="010"
			(-@[red]"0"
			)
		,-@[red][drr] *+[]{\tiny{\txt{110}}}="110"
			(-@[red]"0"
			)
		)
	,(-@[green][dll] *+[]{\tiny{\txt{100 \\ 011}}}
		(-@[red]"100"
		,-@[red][drr] *+[]{\tiny{\txt{111}}}="111"
			(-@[red]"0"
			)
		,-@[red][drrr] *+[]{\tiny{\txt{011}}}="011"
			(-@[red]"0"
			)
		)
	,(-@[green][dl] *+[]{\tiny{\txt{100 \\ 001}}}
		(-@[red]"100"
		,-@[red][drrr] *+[]{\tiny{\txt{001}}}="001"
			(-@[red]"0"
			)
		,-@[red][drrrr] *+[]{\tiny{\txt{101}}}="101"
			(-@[red]"0"
			)
		)
	,(-@[green][d] *+[]{\tiny{\txt{010 \\ 001}}}
		(-@[red]"010"
		,-@[red]"011"
		,-@[red]"001"
		)
	,(-@[green][dr] *+[]{\tiny{\txt{101 \\ 010}}}
		(-@[red]"010"
		,-@[red]"111"
		,-@[red]"101"
		)
	,(-@[green][drr] *+[]{\tiny{\txt{101 \\ 011}}}
		(-@[red]"110"
		,-@[red]"011"
		,-@[red]"101"
		)
	,(-@[green][drrr] *+[]{\tiny{\txt{110 \\ 001}}}
		(-@[red]"110"
		,-@[red]"111"
		,-@[red]"001"
		)
    )
} }
\end{minipage}\hfill
\begin{minipage}{.4\textwidth}
Rank: $r(E)=2$ \\
Independent: all except $E$\\
Bases: all $2$-spaces \\
Circuits: $E$ \\
Hyperplanes: all $1$-spaces \\
Cocircuits: all $2$-spaces \\
Dual: $U_{1,3}$\\
Direct sum: no. \\
Representation: $\left[\begin{array}{ccc} 1 & 0 & \alpha \\ 0 & 1 & \alpha^2 \end{array}\right]_{\mathbb{F}_{2^3}}$
\end{minipage}

\subsubsection{$U_{3,3}$}
\begin{minipage}{.55\textwidth}
\resizebox{\textwidth}{!}{ \xygraph{
[] *+[]{\tiny{\txt{$E$}}}
	(-@[red][dlll] *+[]{\tiny{\txt{100 \\  010}}}
		(-@[red][d] *+[]{\tiny{\txt{ 100}}}="100"
			(-@[red][drrr] *+[]{\tiny{\txt{0}}}="0"
			)
		,-@[red][dr] *+[]{\tiny{\txt{010}}}="010"
			(-@[red]"0"
			)
		,-@[red][drr] *+[]{\tiny{\txt{110}}}="110"
			(-@[red]"0"
			)
		)
	,(-@[red][dll] *+[]{\tiny{\txt{100 \\ 011}}}
		(-@[red]"100"
		,-@[red][drr] *+[]{\tiny{\txt{111}}}="111"
			(-@[red]"0"
			)
		,-@[red][drrr] *+[]{\tiny{\txt{011}}}="011"
			(-@[red]"0"
			)
		)
	,(-@[red][dl] *+[]{\tiny{\txt{100 \\ 001}}}
		(-@[red]"100"
		,-@[red][drrr] *+[]{\tiny{\txt{001}}}="001"
			(-@[red]"0"
			)
		,-@[red][drrrr] *+[]{\tiny{\txt{101}}}="101"
			(-@[red]"0"
			)
		)
	,(-@[red][d] *+[]{\tiny{\txt{010 \\ 001}}}
		(-@[red]"010"
		,-@[red]"011"
		,-@[red]"001"
		)
	,(-@[red][dr] *+[]{\tiny{\txt{101 \\ 010}}}
		(-@[red]"010"
		,-@[red]"111"
		,-@[red]"101"
		)
	,(-@[red][drr] *+[]{\tiny{\txt{101 \\ 011}}}
		(-@[red]"110"
		,-@[red]"011"
		,-@[red]"101"
		)
	,(-@[red][drrr] *+[]{\tiny{\txt{110 \\ 001}}}
		(-@[red]"110"
		,-@[red]"111"
		,-@[red]"001"
		)
    )
} }
\end{minipage}\hfill
\begin{minipage}{.4\textwidth}
Rank: $r(E)=3$ \\
Independent: all\\
Bases: $E$ \\
Circuits: no \\
Hyperplanes: $2$-spaces \\
Cocircuits: $1$-spaces \\
Dual: $U_{0,3}$\\
Direct sum: $U_{1,1}\oplus U_{1,1}\oplus U_{1,1}$. \\
Representation: identity matrix $I_3$
\end{minipage}

\end{document}